\documentclass[11pt,a4paper,reqno,final]{article}
\usepackage{amsmath,amsfonts,amssymb}
\usepackage{amsthm}
\usepackage{cite}
\usepackage{graphicx}
\usepackage{stix}
\usepackage{mathtools}
\usepackage{esint}
\usepackage{multirow}
\usepackage{authblk}
\usepackage{authblk}
\usepackage{tikz}
\usepackage{subcaption}
\usepackage[margin=0.85in]{geometry}
\usepackage[english]{babel}
\usepackage[colorlinks=true,breaklinks=true,linkcolor=lightblue,citecolor=lightgreen,urlcolor=lightblue]{hyperref}
\numberwithin{figure}{section}
\numberwithin{table}{section}
\numberwithin{equation}{section}
\definecolor{lightblue}{rgb}{0.22,0.45,0.70}
\definecolor{lightgreen}{rgb}{0.22,0.50,0.25}
\newcommand{\trinl}{\ensuremath{|\!|\!|}}
\newcommand{\trinr}{\ensuremath{|\!|\!|}}
\newcommand{\mycomment}[1]{}
\newtheorem{thm}{Theorem}[section]

\newtheorem{defn}{Definition}[section]

\newtheorem{lemma}[thm]{Lemma} 
 
\newtheorem{rem}[thm]{Remark}

\newtheorem{cor}[thm]{Corollary}

\newtheorem{example}{Example}[section]
\newcommand\norm[1]{\lVert#1\rVert}

\newcommand{\dx}{\,\mathrm{d}x}

\newcommand{\ds}{\,\mathrm{d}s}
\newcommand{\dt}{\,\mathrm{d}t}
\title{\bf Semi- and Fully-Discrete Analysis of
Lowest-Order Nonstandard Finite Element
Methods for the Biharmonic Wave Problem} 
\author{Neela Nataraj\footnote{\textbf{Corresponding author:} Department of Mathematics, Indian Institute of Technology Bombay, Mumbai, Maharashtra
400076, India, e-mail: neela@math.iitb.ac.in} \quad 
Ricardo Ruiz-Baier\footnote{ School of Mathematics, Monash University, 9 Rainforest Walk, 3800 Melbourne, Victoria, Australia; and Universidad
Adventista de Chile, Casilla 7-D Chillán, Chile, e-mail: ricardo.ruizbaier@monash.edu}
\quad Aamir Yousuf\footnote{ IITB–Monash Research Academy, Indian Institute of Technology Bombay, Mumbai, Maharashtra 400076, India,
e-mail: aamir72@iitb.ac.in}} 
\begin{document}
\maketitle
\noindent
\textbf{Abstract:} This paper discusses lowest-order nonstandard finite element methods for space discretization and explicit and implicit schemes for time discretization of the biharmonic wave equation with clamped boundary conditions. A modified Ritz projection operator defined on $H^2_0(\Omega)$ ensures error estimates under appropriate regularity assumptions on the solution. Stability results and error estimates of optimal order are established in suitable norms for the  semidiscrete and explicit/implicit fully-discrete versions of the proposed schemes. Finally, we report on numerical experiments using explicit and implicit schemes for time discretization and Morley, discontinuous Galerkin, and {C$^0$ interior} penalty schemes for space discretization,  that validate the theoretical error estimates.


\section{Introduction}
\subsection{Scope and Problem Formulation}
The biharmonic wave model finds application in representing various physical phenomena, such as the bending of plates, and elasticity in thin plates under dynamic loading conditions. This is also a foundational part for more advanced linear and nonlinear models such as Kirchhoff-type equations using Monge--Amp\`ere forms  \cite{MR4403922,MR1422248,MR0953313}, which in turn serve as model for the transmission of waves, damping phenomena, standing waves, properties of ferromagnetic materials, and vibration modes within plates.

This paper analyses lowest-order nonstandard finite element methods (FEMs) for space discretization and {explicit leapfrog/implicit Newmark time discretization schemes for the  biharmonic wave equation} 
\begin{equation}
u_{tt} + \Delta^2 u = f(x,t), \quad (x,t) \in \Omega \times \left(0,T \right], \label{P1 strong_form}
\end{equation}
with initial and clamped boundary conditions
\begin{align}
\begin{cases}
\;u(x,0)=u^0(x),&\\
\hspace{0.01cm}u_t(x,0)=v^0(x) &\quad\text{in }\Omega,\\
 \; u=\frac{\partial u}{\partial n} =0, &\quad\text{on }\partial \Omega \times \left(0,T \right].\label{P1 strong_icbc}
  \end{cases}
\end{align}

Here $\Omega$ is a bounded polygonal Lipschitz domain in $\mathbb{R}^2$ with  boundary $\partial \Omega$ and outward-pointing unit normal $n$, 
$$\Delta^2 u:= \frac{\partial^4 u}{\partial x^4} + 2 \:\frac{\partial^4 u}{\partial x^2 \partial y^2} +\frac{\partial^4 u}{\partial y^4}$$
            denotes the biharmonic operator, $\partial u / \partial n = \nabla u \cdot n $ is the outer normal derivative of $u$ on $\partial \Omega$,   $u_{t} $, $u_{tt} $ denote the first- and second-order  derivatives with respect to time, respectively, {and $f\in L^2(L^2(\Omega))$ is a given source function}.  

\subsection{Literature Overview}
Despite being a classical problem, research on the biharmonic problem remains an area of active interest, as evidenced by recent contributions \cite{MR4092275,MR4376276,MR4485999,MR3767813,MR4597463,MR4480625}. {The design of}  \(C^1\) finite elements, which ensure continuity of both basis functions and their first-order derivatives over the closure of the domain and conform to the space \(H^2(\Omega)\), for fourth-order problems is notoriously challenging. Consequently, nonstandard FEMs such as the Morley FEM \cite{MR4230429,MR3934690,MR2207619}, discontinuous Galerkin (dG) FEM \cite{MR2755946,DG-BI}, and the \(C^0\) interior penalty (\(C^0\)IP) method \cite{MR4481121,MR2298696} present attractive alternatives. 
Recently, an abstract framework for analyzing lowest-order FEMs for the clamped plate biharmonic problem has been established in \cite{MR4376276}. {One aim of this paper is to investigate and adapt such an analysis for the case of biharmonic wave problems.

Fourth-order parabolic problems have been addressed using the Morley FEM \cite{MR4305273}, dG FEM \cite{MR3356023}, and the \(C^0\) interior penalty method \cite{MR3023298}, typically employing the backward Euler method for time discretization. Similar numerical methods have been extensively studied for second-order hyperbolic equations, as seen in \cite{MR3438428,MR2272600,MR0945124,MR4456271,MR2942377,MR1241479,MR0349045,MR2052507}. However, despite their importance in applications, the literature on the numerical approximation of the fourth-order biharmonic wave equation is relatively sparse. Notable contributions include \cite{MR4444402}, where the biharmonic wave problem is discretized using classical \(C^1\)-conforming Bogner--Fox--Schmit elements in space, combined with Galerkin and collocation techniques for time discretization. A mixed velocity-moment formulation is analyzed in \cite{MR2114953} for the fourth-order Kirchhoff--Love dynamic plate equations using Lagrange FEs in space with both explicit and implicit central difference schemes in time. Additionally, error estimates for the fourth-order wave equation have been derived using a combination of discrete Galerkin and second-order accurate methods for space and time discretizations \cite{MR0317559}. Mixed FEMs for fourth-order wave equations with various boundary conditions and optimal error estimates have been studied in \cite{MR3003381,he23}.
\subsection{Specific Contributions}
We now describe the main contributions of this paper.
\begin{itemize}
\item To the best of our knowledge, this is the {\it first} attempt to analyze the biharmonic wave equation for nonstandard {FEM for spatial and the  {explicit}/implicit} schemes for time discretization. 

\item {The Courant–Friedrichs–Lewy (CFL) condition for wave equation is extensively discussed in the literature (see, e.g., \cite{MR2511734}), but its analysis specifically targeted for the biharmonic wave equation was still {\it unavailable}.
The CFL condition for explicit time discretization scheme applied to $u_{tt}+ \Delta u =f$ (resp. $u_{tt}+ \Delta^2 u =f$) reads \(
k \le C_{\text{CFL}} h
\) (resp. \(
k \le C_{\text{CFL}} h^2
\) ), where \( h \) represents the maximal mesh-size of  a quasi-uniform triangulation of the spatial domain. Here 
\(
{C_{\text{CFL}} := \beta^{-1/2}C^{-1}_{\text{inv}} },
\)
where \( C_{\text{inv}} \) is the constant from discrete inverse inequality applied to piecewise $H^1$ (resp. $H^2$) functions and \( \beta \) is the continuity constant corresponding for the bilinear form in the weak formulation for wave (resp. biharmonic wave) equations.  
One key observation is that {the form of constant that appears in the CFL condition} is invariant in the discretization 
using explicit scheme for temporal discetization.}
\item {The CFL condition also shows that explicit schemes applied to higher-order equations (in space) face stringent stability and convergence restrictions due to the conditions on the time step and motivates the {\it implicit scheme} for temporal discretization discussed in this article.
 In comparison to the explicit scheme, the {\it implicit scheme} discussed in this article (a) removes the necessity for the quasi-uniformity assumption on the mesh and (b) relaxes the constraints imposed by the Courant--Friedrichs--Lewy  (CFL) condition. }
\item The regularity results advanced in this work (see  Lemma~\ref{P1 regularity_lemma}) are established under (a) certain smoothness assumptions on $f$ and its derivatives and (b) for non-homogeneous initial conditions, which are different from those in \cite[Theorem 7.1]{MR0350178}.  This is attained by using the approach of explicit solution representation {for proving} the regularity results rather than using the energy arguments as discussed in \cite{MR0350178}.

\item We employ a modified Ritz projection (see the definition of $\mathcal{R}_h$ in \eqref{P1 ritz_projection}) on {$H^2_0(\Omega)$ defined with} the help of the companion operator (cf. \cite{MR4376276}). 
It is also important to note that {this modified Ritz projection readily yields $L^2-$estimates (see Corollary~\ref{corr}, Remark~\ref{newremark} for the fully-discrete schemes) based on the energy norm error estimate of the solution and $L^2-$error estimate of time derivative of the solution without the need for any additional analysis.}

\item 
The test function in the continuous weak form is selected following an approach that consists {in} lifting discrete functions to $H^2_0(\Omega)$ using a suitably defined {\it smoother} \cite{MR4376276}. This {\it novel} technique {in the context of evolution problems} helps to bound the semidiscrete and fully-discrete errors by manipulating the continuous and discrete formulations. This is a significant improvement in comparison to more standard approaches that assume higher regularity of the continuous solution and in which the PDE \eqref{P1 strong_form} is tested against a function from the discrete space. Moreover, the new approach facilitates a more elegant error analysis avoiding extra boundary terms that arise in the standard error analysis typically used in the literature, see for example, \cite{MR2272600}. 
\item 
 Given the regularity of the exact solution $u$, as discussed in Lemma~\ref{P1 regularity_lemma}, quasi-optimal convergence is achieved in both the {maximum error in energy and the $L^2-$norm over a finite time interval}. For the semidiscrete scheme, we show convergence rates of ${\mathcal{O}}(h^{\gamma_0})$ and ${\mathcal{O}}(h^{2\gamma_0})$, respectively, where {$\gamma_0 \in (1/2, 1]$ is the index of elliptic regularity of the biharmonic operator} and $h$ represents the spatial mesh-size. On the other hand, for the explicit/implicit fully-discrete schemes,  the convergence rates are ${\mathcal{O}}(h^{\gamma_0}+k^2)$ and ${\mathcal{O}}(h^{2\gamma_0}+k^2)$, in energy and $L^2$ norms, where $k$ represents the time step. 
 \end{itemize}
\subsection{Outline of the Paper}
The contents of this paper are organized as follows. In the remainder of this section, we introduce standard notations to be used throughout the manuscript,  provide the weak formulation for the problem, and state the regularity results of the continuous solution under smoothness assumptions on the {given data.} Section \ref{P1 semi_errors_section} discusses the semidiscrete FE approximation and error analysis. Sections \ref{P1 fully_discrete_section} and \ref{P1 implicit_direct} are devoted to the error analysis for explicit and implicit fully-discrete schemes.  Numerical results obtained with Morley, $C^0$IP, and dGFEM for space discretization  and explicit/implicit schemes for time discretization are presented in Section \ref{P1 numeric_section}  to validate the theoretical error bounds and also to illustrate the use of the method in a simple application problem in heterogeneous media.

\subsection{Preliminaries and Functional Setting}
\bigskip 
\paragraph{{Notation.}} For ${\cal X} \subset \mathbb{R}^2$,  we denote the Sobolev space $W^{m,2}({\cal X})$ by $H^m({\cal X})$  and equip it with the norm 
$$\norm{w}_{H^m({\cal X})}= \Big(\underset{{|i| \le m}}{\sum} \norm{D^iw}_{L^2({\cal X})}^2\Big)^{1/2}$$
and semi-norm 
$$| w |^2_{H^m({\cal X})}=\Big(\underset{{|i| = m}}{\sum}\|D^iw \|_{L^2({\cal X})}^2\Big)^{1/2}.$$
For simplicity, we denote the $L^2$-inner product by $(\cdot, \cdot)$ and norm by $\| \cdot \|.$ {Throughout this paper, $\mathcal{T}$ denotes a shape-regular triangulation of $\Omega$ unless mentioned otherwise,} $H^m({\cal T})$ denotes the Hilbert space $\underset {{K \in {\cal T}}}{\prod}H^m (K)$, and {$P_r({\cal T})$ the} space of globally $L^2$-{functions, which} are polynomials of degree at most $r$ in each $K$. The piecewise energy norm is denoted by $\trinl \cdot \trinr_{\text{pw}}:=|\cdot|_{H^{2}({\cal T})}$ and $D^2_\text{pw}$ stands for the piecewise Hessian.

\noindent 
Let $X$ be a normed space with norm $\|\cdot\|_X$ and $g:(0,T) \rightarrow X$ be a measurable function. Then for $1 \le p \le \infty$, we recall that  
\[\displaystyle \norm{g}^p_{L^p(X)}:=\int_0^T \norm{g(t)}_X^p \dt, \ 1\le p< \infty \quad \text{ and }\quad 
 \norm{g}_{L^\infty(X)}:=\underset{0 \le t \le T}{\text{ess sup}} \norm{g(t)}_X.\]
Let 
$L^p(0,T;X):=\left\{g:(0,T) \rightarrow X: \norm{g}_{L^p(X)}<\infty\right\}$
and $C^k([0,T];X)$ denote all $C^k$ functions $h :[0,T] \rightarrow X$ with 
$$\norm {h}_{C^k([0,T];X)}=\underset{0 \le i \le k}{\sum} \underset{0 \le t\le T}{\max} {\norm{h^i(t)}} < \infty, \text{ where }h^i(t)=\frac{\partial^i h}{\partial t^i}.$$

For  real numbers $a > 0$, $b > 0$, and $\epsilon > 0$, we will make repeated use of the  weighted {Young's} (arithmetic-geometric mean) inequality $ab \leq \frac{\epsilon}{2}a^2 + \frac{1}{2\epsilon}b^2$. Finally, as usual, the notation $a \lesssim b$ represents $a \le Cb$, where the generic constant $C$ is independent of both mesh-size and time discretization parameter.

\paragraph{Weak Formulation.} Let $a(\cdot,\cdot):H^2_0(\Omega) \times H^2_0(\Omega) \rightarrow {\mathbb R}$ be a  symmetric, continuous, and $H^2_0(\Omega)-$elliptic bilinear form defined by 
$ \displaystyle a(w,v)=\int_\Omega D^2w:D^2v \dx.$
The weak formulation that corresponds to \eqref{P1 strong_form}-\eqref{P1 strong_icbc}  {seeks, for $t \in (0,T)$, a $u(t) \in H_0^2(\Omega)$} such that 
\begin{align}
   \begin{split}  (u_{tt},v)+a(u,v)&= \left(f,v\right)  \text{ for all }v \in H_0^2(\Omega), \label{P1 weak_form}\\  
    u(0)&=u^0 \text{ and } u_t(0)=v^0.
     \end{split}
\end{align}
{For notational simplicity, the dependency of functions on $t$ will be skipped in some instances (whenever there is no chance of confusion); for example, $u:=u(t)$ in \eqref{P1 weak_form}, etc.}  
Given $f \in L^2(L^2(\Omega))$, $u^0 \in H^2_0(\Omega)$, and $v^0 \in L^2(\Omega)$, there exists a unique $u \in C^0([0,T];H^2_0(\Omega)) \cap  C^1([0,T]; L^2(\Omega)) $ satisfying \eqref{P1 weak_form} (see \cite[ p. 93]{MR0350178}).
 
\paragraph{Regularity.}
It is well-known (see, e.g., \cite[p. 761]{MR1115240}) that the eigenvalue problem
\begin{align}
\begin{cases}
    \Delta^2 \psi = \lambda \psi \qquad\qquad\;&\text{ in }\Omega,\\
     \hspace{0.05cm}\quad\psi =0,\ \frac{\partial \psi}{\partial n}=0 \;\;\quad&\text{ on } \partial \Omega,  \label{P1 Ev}
    \end{cases}
\end{align}
admits an increasing  sequence  of eigenvalues $\{\lambda_n \}_{n=1}^{\infty}$ with  $ 0<\lambda_1 \le \lambda_2 \le \cdots \le \lambda_n \rightarrow \infty \text{ as } n \rightarrow \infty$ and the corresponding family of eigenfunctions $\{\psi_n\}_{n=1}^{\infty}$ form an  orthonormal basis of $L^2(\Omega)$. Let us define an unbounded operator $(A, D(A))$ in $L^2(\Omega)$ by
$D(A) = \{ \psi \in H_0^2(\Omega) : \Delta^2 \psi \in L^2(\Omega)\} \text{ and }A \psi = \Delta^2 \psi$ for all $\psi \in D(A)$. Further, for any $r \in \mathbb {R}^+\cup \{0\}$, we define 
\begin{equation*}
    A^r w = \sum_{n=1}^ \infty \lambda_n^r(w, \psi_n)\psi_n \quad \text{and}\quad   D(A^r)= \bigg\{w= \sum_{n=1}^{\infty}w_n \psi_n: w_n \in \mathbb{R}, \ \sum_{n=1}^\infty \left|\lambda_n w_n \right|^{2r} < \infty \bigg\}. 
\end{equation*}
Moreover, the space $D(A^r)$ is equipped with the norm 
\begin{equation}
    \norm{w}_{D(A^r)}=\Big(\sum_{n=1}^ \infty \lambda_n^{2r}\left|(w, \psi_n)\right|^2 \Big)^{1/2}. \label{P1 norm_DA}
\end{equation}
In particular, $\|\cdot\|_{D(A^0)}=\|\cdot\|$ and $D(A^{1/2})=H^2_0(\Omega)$ and $D(A)\subset H^2_0(\Omega) \cap H^{2+\gamma_0}(\Omega)$. {Here and throughout the paper,  
 $\gamma_0\in(1/2,1]$ is the index of elliptic regularity of the biharmonic operator  (see further details in, e.g., \cite{MR0595625}).}

The next lemma presents regularity results for the continuous solution that are employed for the semidiscrete and fully-discrete error estimates proposed in this article. {We note that the regularity results for fourth-order wave equation with homogeneous boundary and initial conditions available in {\cite[Theorem 7.1]{MR0350178}} 
assume a higher smoothness for $f$ and its temporal derivatives. {More precisely, for any positive integer $l$,  $f \in L^2(H^{2(2l-1)}(\Omega)),$ {$\frac{\partial^{2l} f }{\partial t^{2l}}\in L^2(L^2(\Omega))$, where $\frac{\partial^{l} f}{\partial t^{l}} $} denotes the $l^{th}$ partial derivative of $f$ with respect to $t$,  and {$f|_{t=0}=\frac{\partial f}{\partial t}|_{t=0} = \cdots = \frac{\partial^{2l-1} f}{\partial t^{2l-1}}|_{t=0}=0$}, the same reference shows that   $u \in L^2(H^{2(2l+1)}(\Omega)) $ and ${\frac{\partial^{2 l+1} u }{\partial t^{2l +1}}} \in L^2(L^2(\Omega))$, with positive integer $l$.} In contrast, the regularity results in this article are established under (a) different smoothness assumptions on $f$ and its derivatives and (b) for non-homogeneous initial conditions. The proof of Lemma~\ref{P1 regularity_lemma}  presented in the Appendix is based on a solution representation approach. 
\begin{lemma}[Regularity]\label{P1 regularity_lemma} {Let  $\gamma_0 \in (1/2,1]$ be elliptic regularity index of the biharmonic operator}. The assumptions 
\begin{align*}
& u^0 \in  D(A^2), \ v^0 \in  D(A^{3/2}), \text{ and } 
f(0) \in  D(A), \; f_t(0) \in D(A^{1/2}),\; f_{tt}(0)\in L^2(\Omega),\\
 &\qquad\qquad\qquad\qquad\qquad\quad \;f,f_t,f_{tt},f_{ttt} \in L^2(L^2(\Omega))
\end{align*}
{guarantee that the norms 
$\norm{u}_{L^\infty(H^{2+\gamma_0}(\Omega))},$   $\norm{u_t}_{L^\infty(H^{2+\gamma_0}(\Omega))}, $ $\norm{u_{tt}}_{L^\infty(L^2(\Omega))}, $ $\norm{u_{tt}}_{L^\infty(H^{2+\gamma_0}(\Omega))}$, $ \norm{u_{ttt}}_{L^\infty(L^2(\Omega))}$,  and $\norm{u_{tttt}}_{L^\infty(L^2(\Omega))}$  are bounded.}
\end{lemma}
\begin{rem}[Initial Data]
    {The approximation of the initial conditions for the semidiscrete scheme in \eqref{P1 semi_formulation} demands  that the initial data $u^0$ and $v^0$  {belong} to $H_0^2(\Omega)$. For the fully-discrete schemes, this is further relaxed to $u^0 \in H_0^2(\Omega)$ and  $v^0 \in L^2(\Omega)$.}
\end{rem}

\section{Semidiscrete  Error Analysis}\label{P1 semi_errors_section}
An outline of this section is as follows. Subsection \ref{P1 semi_schemes} describes  lowest-order FE schemes for the spatial variable. The Morley interpolation operator, the companion operator, and a modified Ritz projection operator are essential tools for the analysis and are discussed in Subsection \ref{P1 ritz_projection_section}. Finally, semidiscrete error estimates are derived in Subsection \ref{P1 semi_errors_section1}. 
\subsection{Space Discretization}\label{P1 semi_schemes}
 For a generic triangle $K \in {\cal T}$ of the shape-regular triangulation $\mathcal{T}$ of $\bar{\Omega}$,  let $h_K$ denote its diameter, $|K|$ its area, $n_K$  be outward unit normal along $\partial K$,  and let $h:= \underset {K \in \mathcal{T}}{{\max }}h_K $. Let ${\mathcal V}(\Omega)$ (resp. ${\mathcal V}(\partial \Omega))$ denote the set of all interior (resp. boundary) vertices of  ${\cal T}$ and let ${\mathcal V}= {\mathcal V}(\Omega) \cup {\mathcal V}(\partial \Omega)$.  Let ${\mathcal E}(\Omega)$ (resp. ${\mathcal E}(\partial \Omega)$) denote the set of all interior (resp. boundary)  {edges} of  ${\cal T}$ and let ${\mathcal E}= {\mathcal E}(\Omega) \cup {\mathcal E}(\partial \Omega)$. For any edge $e \in {\mathcal E}$, we define its edge patch ${\omega(e)}$ by $ \text{int} \: (K_+ \cup K_-)$ if $e =\partial K_+ \cap \partial K_-  \in {\cal E}(\Omega)$ and $\text{int}\: (K) \text{ if } e \in {\cal E}(\partial \Omega)$. Let $K_+$ and $K_-$ be adjacent triangles with unit normal vector $n_{K_+}|_e =n|_e=-n_{K_-}|_e$ along edge $e$ pointing outside from $K_+$ to $K_-$. Define jump of a function $\varphi$, $ \big[\!\!\big[\varphi\big]\!\!\big]$ by  $  \varphi |_{K_{+}}-  \varphi |_{K_{-}}\text{ if } e=\partial K_+ \cap \partial K_-  \in {\cal E}(\Omega)$ and $\varphi|_e  \text{ if } e \in {\cal E}(\partial \Omega)$. Further, define average $\big\{\!\!\!\big\{\varphi\big\}\!\!\!\big\}$ by $\frac{1}{2} (\varphi |_{K_{+}}+  \varphi|_{K_{-}} )\text{ if } e=\partial K_+ \cap \partial K_-  \in {\cal E}(\Omega)$ and  $\varphi|_e  \text{ if } e \in {\cal E}(\partial \Omega)$.

\medskip  
After defining the space 
\begin{align*}
{\rm M}'({\cal T}):=\{& {v_{\rm M}} \in P_2({\cal T}): 
 {v_{\rm M}} \text{ is continuous}
\text{ at interior vertices and its normal derivatives}\\
&\text{ are continuous at the midpoints of interior edges} \},\end{align*}
we recall from \cite{MR4376276} the definition of the nonconforming Morley FE space ${{\rm M}({\cal T})}$ below 
\begin{align*}
{{\rm M}({\cal T})}:=\{& {v_{\rm M}} \in {\rm M}'({\cal T}):
 {v_{\rm M}} \text{ vanishes at the vertices of }\partial \Omega
\text{ and its normal derivatives}\\
&\text{ vanish at the midpoints of boundary edges of  $\partial \Omega$} \}.
\end{align*}
Let $V_h \subset {\cal H}^2({\cal T})$ be a finite-dimensional subspace and let $a_h(\cdot, \cdot):(V_h+{\rm M}({\cal T})) \times (V_h+ {\rm  M}({\cal T})) \rightarrow \mathbb{R}$ be a symmetric, continuous, and elliptic bilinear form with respect to a mesh-dependent (broken) norm on $V_h$ defined by 
\begin{equation}
\norm{v_h}_{h}^2=\sum_{K\in \mathcal{T}}|v_h|_{H^2(K)}^2 + \sum_{e\in\mathcal{E}}\sum_{z\in\mathcal{V}(e)}h_e^{-2}\big|\big[\!\!\big[   v_h\big]\!\!\big](z)\big|^2+ \sum_{e\in\mathcal{E}}\Big| \intbar_{e} \Big[\!\!\Big[  \frac{\partial v_h}{\partial n}\Big]\!\!\Big]\ds \Big|^2, \label{P1 mesh_norm}
\end{equation}
see, e.g.,  \cite{MR3371900}. In other words, there exist $ \alpha, \beta >0$ such that for all $w_h, v_h \in V_h$
\begin{equation}
    a_h(w_h,v_h)=a_h(v_h,w_h), \quad \alpha  \|w_h\|_h^2 \le a_h(w_h,w_h), \quad a_h(w_h,v_h) \le \beta \|w_h\|_h\|v_h\|_h.\label{P1 a_h_properties}
\end{equation}  
Note that the continuity in \eqref{P1 a_h_properties} holds in $V+ V_h$. The semidiscrete formulation that corresponds to \eqref{P1 weak_form} seeks $u_h(\cdot,t):(0,T] \rightarrow V_h $ such that
\begin{align}
\begin{split}
    (u_{htt},v_h )+a_h(u_h,v_h)& =(f,v_h) \text{ for all }v_h \in V_h, \\  
  { u_h(0)={\cal R}_hu^0} &\text{ and }
   {u_{ht}(0)={\cal R}_hv^0,}
   \end{split}
      \label{P1 semi_formulation}
\end{align}
{where ${\cal R}_h$ is a {\it modified Ritz projection} defined in \eqref{P1 ritz_projection}.} {Note that the definition of the semidiscrete formulation assumes that the initial data $u^0$ and $v^0$ belongs to $H^2_0(\Omega)$. }

\medskip 
Examples of some lowest-order FEMs with choices for $V_h$, $a_h(\cdot,\cdot)$, and the corresponding norms are given below.

\medskip 
 \begin{example}[Morley FEM] \label{ex1}
We can choose $V_h:=\rm{M({\cal T})} $ and define the discrete bilinear form 
\begin{equation}
    a_h(w_h,v_h):=a_{\text{\rm pw}}(w_h,v_h):= \int_\Omega D^2_{{\rm pw}}w_h:D^2_{\rm{pw}}v_h \dx. \label{P1 Morley}
\end{equation}
For all $v+v_h \in V+  {\rm M}({\cal T})$, the discrete norm reads $\trinr v+ v_h\trinr_{\text{\rm pw}}:=a_{\rm pw }(v + v_h, v+ v_h)^{1/2}$.
\end{example}
\begin{example}[dG FEM]
Let us now choose $V_h:=P_2({\cal T})$ and define the discrete bilinear form 
\begin{equation*}
a_h(w_h,v_h):=a_{\text{\rm pw}}(w_h,v_h) + b_h(w_h,v_h)+c_{\text{\rm dG}}(w_h,v_h), 
\end{equation*}
where $a_{\text{\rm pw}}(w_h,v_h)$ is defined  in \eqref{P1 Morley}, and for $v_h,w_h \in P_2({\cal T})$, $b_h(\cdot,\cdot)$ and $c_{\text{\rm dG}}(\cdot,\cdot)$ are defined by 
\begin{subequations}
\begin{align}
    b_h(w_h,v_h)& :=-\sum_{e \in {\cal E}} \int_e \big[\!\!\big[\nabla w_h\big]\!\!\big] \cdot \big\{\!\!\!\big\{D^2_{\rm{pw}}v_h\big\}\!\!\!\big\}n \ds-\sum_{e \in {\cal E}} \int_e \big[\!\!\big[\nabla v_h\big]\!\!\big] \cdot \big\{\!\!\!\big\{D^2_{\text{\rm pw}}w_h\big\}\!\!\!\big\}n \ds, \label{P1 bh} \\
    c_{\text{\rm dG}}(w_h,v_h)& :=\sum_{e \in {\cal E}} \frac{\sigma^1_{\rm{dG}}}{h_e^3} \int_e\big[\!\!\big[ w_h\big]\!\!\big]  \big[\!\!\big[v_h\big]\!\!\big]\ds  +\sum_{e \in {\cal E}} \frac{\sigma^2_{\rm{dG}}}{h_e} \int_e  \Big[\!\!\Big[\frac{\partial w_h}{\partial n} \Big]\!\!\Big] \Big[\!\!\Big[\frac{\partial v_h}{\partial n} \Big]\!\!\Big]\ds, \label{P1 cdg}
    \end{align}\end{subequations}
where $\sigma^1_{\rm{dG}}, \sigma^1_{\rm{dG}}>0$ are  penalty parameters. 
The dG norm $ \| \cdot\|_{\rm{dG}}$ on $V_h$ is defined by 
\[ \| v_h\|_{\rm{dG}}:=\left(\trinr v_h\trinr_{\rm{pw}}^2+c_{\rm{dG}}(v_h,v_h)\right)^{1/2}.\] 
\end{example}
\begin{example}[C$^0$IP] \label{ex3}  Choose $V_h:= P_2({\cal T}) \cap H^1_0(\Omega)$ and define 
 $   a_h(\cdot,\cdot):=a_{\rm{pw}}(w_h,v_h) + b_h(w_h,v_h)+c_{\rm{IP}}(w_h,v_h) $,
where $a_{\rm{pw}}(\cdot,\cdot),\;b_h(\cdot,\cdot)$ are defined in  \eqref{P1 Morley}, \eqref{P1 bh}, respectively, and for $v_h, w_h \in P_2({\cal T}) \cap H^1_0(\Omega) ,$ we define 
\begin{equation*}
\displaystyle c_{\rm{IP}}(w_h,v_h) :=   \sum_{e \in {\cal E}} \frac{\sigma_{\rm{IP}}}{h_e} \int_e  \Big[\!\!\Big[\frac{\partial w_h}{\partial n} \Big]\!\!\Big] \Big[\!\!\Big[\frac{\partial v_h}{\partial n} \Big]\!\!\Big]\ds, 
\end{equation*}
with the penalty parameter $\sigma_{\rm{IP}}>0$. 
The norm $\|\cdot\|_{\rm{IP}}$ on $P_2({\cal T}) \cap H^1_0(\Omega)$ is defined by 
\[\| v_h\|_{\rm{IP}}=\left(\trinr v_h\trinr_{\rm{pw}}^2+c_{\rm{IP}}(v_h,v_h)\right)^{1/2}.\]
\end{example}
\begin{rem}[Properties of discrete bilinear forms \cite{MR4376276}] The three norms $\|\cdot \|_{\rm{pw}}$,  $\|\cdot \|_{\rm{dG}}$, and  $\|\cdot \|_{\rm{IP}}$ are equivalent to the norm $\|\cdot\|_h$ defined in \eqref{P1 mesh_norm} and the bilinear form $a_h(\cdot,\cdot)$ defined in each case is symmetric, elliptic, and continuous with respect to $\|\cdot\|_h$, for sufficiently large  penalty parameters for the dG and $C^0$IP schemes. 
\end{rem}
\subsection{Modified Ritz Projection }\label{P1 ritz_projection_section}
 The error control in the semidiscrete and fully-discrete approximation is estimated with the help of the Ritz projection operator ${\cal R}_h$ defined from $H^2_0(\Omega)$ to the conforming finite dimensional space ~\cite[Chapter 1]{thomee}. However, the standard definition  
 \begin{equation}
     a_h({\cal R}_h w, v_h) := a(w,v_h) \qquad \text{for all $v_h \in V_{h}$}, \label{rtz}
     \end{equation}
 does not hold for   $v_h \in V_h \subset H^2({\mathcal T})$ for the nonstandard schemes discussed in this paper since the bilinear form $a(\cdot,\cdot)$ may not be defined for functions in $V_h$.  
 
 \medskip
 The alternative ideas that define Ritz projections for nonconforming methods (see for example,   \cite{MR4305273} for the fourth-order nonlinear parabolic extended Fisher--Kolmogorov equation), typically assume higher regularity of the solution in space. In this article we resolve this issue by means of a modified Ritz projection (see (\ref{P1 ritz_projection}) below) that employs a smoother $Q: V_h \rightarrow H^2_0(\Omega)$. Such a smoother $Q$ is defined as $JI_{\rm M}$, where $J$ (resp. $I_{\rm M}$) is the companion (resp. extended Morley interpolation)  operator defined in Lemma~\ref{bhcompanion_lem} (resp. Definition~\ref{bhmorley_lem}) below. 

\begin{defn}[Morley Interpolation
\cite{MR4376276}]\label{bhmorley_lem}
For all $v_{\rm{pw}}\in H^2(\mathcal T)$, the extended Morley interpolation operator $I_{\rm M}:H^2({\mathcal T} ) \rightarrow {\rm M}({\mathcal T} )$ is defined by 
\[ 
 (I_{\rm {M}}v_{\rm{pw}})(z):= 
 |\mathcal T(z)|^{-1}
\sum_{K \in {\mathcal T}(z)} (v_{\rm{pw}}|_K)(z)\quad \text{and}\quad \displaystyle \fint_e\frac{\partial (I_{\rm M} v_{\rm{pw}})}{\partial n} \ds :=  \fint_e \Big\{\!\!\!\Big\{\frac{\partial v_{\rm{pw}}}{\partial n}\Big\}\!\!\!\Big\} \ds .\]
In case of an interior vertex $z$, ${\mathcal T}(z)$ represents the collection of attached triangles, and $|{\mathcal T}(z)|$ indicates the number of such triangles connected to vertex $z$.  
\end{defn}
\begin{lemma}[Companion Operator and Properties \cite{carsput2020,MR4376276}]\label{bhcompanion_lem}
Let ${\rm HCT}(\mathcal{T})$ denote the Hsieh--Clough--Tocher FE. 
There exists a linear mapping $J: {\rm M}(\mathcal{T})\to ({\rm HCT}(\mathcal{T})+P_8(\mathcal{T})) \cap H^2_0(\Omega)$ such that any $w_{\rm M}\in {\rm M}(\mathcal{T})$ satisfies
\begin{align}
Jw_{\rm M}(z)&=w_{\rm M}(z) \qquad \qquad\qquad\qquad\qquad \text{ for } z\in\mathcal{V}, \nonumber \\ 
 \nabla ({J}w_{\rm M})(z)&=
|\mathcal{T}(z)|^{-1}\sum_{K\in\mathcal{T}(z)}(\nabla w_{\rm M}|_K)(z)
\quad\; \text{ for }z\in\mathcal{V}(\Omega),  \nonumber\\
 \fint_e \frac{\partial J w_{\rm M}}{\partial n} \ds&=\fint_e \frac{ \partial w_{\rm M}}{\partial n} \ds \qquad\qquad\qquad\qquad \text{ for any } e\in\mathcal{E},\nonumber\\ 
 \trinl w_{\rm M}- J w_{\rm M} \trinr_{\rm pw} &\lesssim  \min_{v\in  H^2_0(\Omega)}  \trinl w_{\rm M}- v \trinr_{\rm pw}, \nonumber \\
 \|{v_h -Qv_h}\|_{H^s({\cal T})} &\le C_1 h^{2-s}\min_{v\in  H^2_0(\Omega)}\norm{v- v_h}_h\quad\quad\text{ for }C_1>0 \text{ and }0 \le s \le 2.  \label{P1 norm_Q}
\end{align} 
\end{lemma} 
\noindent
Next, we define the {\it modified} Ritz projection $\mathcal{R}_h: H^2_0(\Omega)\rightarrow V_h $  as follows: 
\begin{equation}
    a_h({\cal R}_hw,v_h )=a(w, Qv_h )\qquad  \text{ for all } v_h \in V_h,\, w \in H^2_0(\Omega)\label{P1 ritz_projection}.
\end{equation}
The approximation properties of $\mathcal{R}_h$ in the piecewise energy and $L^2-$norms hold under the sufficient conditions in {\bf (H1)}-{\bf (H5)} listed below (see also \cite{MR4376276}): 
\begin{description}
\item[(H1)] All $v_{\rm{M}} \in {\rm{M(\cal T)}}$, $ w_h \in V_h$, and $v ,w\in H^2_0(\Omega)$ satisfy
\begin{align*}
a_{\rm{pw}}(v_{\rm{M}}, w_h - I_{\rm{M}} w_h) + b_h(v_{\rm{M}}, w_h- I_{\rm{M}} w_h) \lesssim
\trinl v_{\rm{M}}-v \trinr_{\rm{pw}} \|w_h -w\|_h.    \end{align*}
\item[(H2)] All $v_{\rm{M}},w_{\rm{M}} \in {\rm{M(\cal T)}}$ and $ v_h, w_h \in V_h$ satisfy
\begin{align*}
 b_h(v_{\rm{M}},w_{\rm{M}})&=0 \\ 
b_h(v_h + v_{\rm{M}}, w_h+ w_{\rm{M}}) &\lesssim \|v_h+v_{\rm{M}}\|_h \|w_h+ w_{\rm{M}}\|_h.
\end{align*}
\item[(H3)] For all $v_h, w_h \in V_h$  and $ v,w \in H^2_0(\Omega)$, there holds 
\[c_h(v_h, w_h) \lesssim \| v - v_h\|_h  \|w- w_h\|_h.\]
\item[(H4)] All  $v_h \in V_h$, $w_{\rm{M}} \in {\rm{M(\cal T)}}$, {and  $v,w \in H^2_0(\Omega)$} satisfy  
\begin{equation*}
 a_{\rm{pw}} ( v_h -I_{\rm{M}} v_h, w_{\rm{M}})+b_h( v_h, w_{\rm{M}})\lesssim \| v-v_h \|_h   \trinl w- w_{\rm{M}} \trinr_{\rm{pw}} . 
\end{equation*}
\item[(H5)] All  $v_{\rm{M}} \in {\rm{M(\cal T)}}$, $w_h \in V_h$, {and  $v,w \in H^2_0(\Omega)$} satisfy  
\begin{equation*}
 a_{\rm{pw}} ( v_{\rm{M}}, w_h -I_{\rm{M}} w_h)+b_h(  v_{\rm{M}},w_h -I_{\rm{M}} w_h )\lesssim   \trinl v- v_{\rm{M}} \trinr_{\rm{pw}} \| w-w_h \|_h  . 
\end{equation*}
\end{description}
The lowest-order methods listed in Examples~\ref{ex1}-\ref{ex3} satisfy {\bf (H1)}-{\bf (H5)}, see \cite[Sections 7-8]{MR4376276}. 
\begin{lemma}[Approximation Properties \cite{MR4376276}]\label{P1 ritz_lemma} For any $v\in H^2_0(\Omega),$ let ${\mathcal R}_h v \in {\rm M}({\mathcal T})$ be its Ritz's projection defined in \eqref{P1 ritz_projection} and suppose that the hypotheses {\bf (H1)}--{\bf (H5)} hold. 
Then there exists  $C_2>0$ such that 
\begin{equation}
   { \|{v-\mathcal{R}_hv}\| + h^{\gamma_0}\norm{v-\mathcal{R}_hv}_h \le C_2 h^{2\gamma_0}\norm{v}_{H^{2+\gamma_0 }(\Omega)}.}  \label{P1 norm_ritz}
\end{equation}
\end{lemma}
\subsection{Error Estimates} \label{P1 semi_errors_section1}
The stability and error estimates for the semidiscrete scheme  {are} presented in this {subsection.} The proofs herein employ the following useful result. 
\begin{lemma}[{Gronwall's} Lemma \cite{chen}]\label{P1 gronwall}
Let $g$, $h$, and $r$ be non-negative integrable functions on $[0,T]$ and let $g$ satisfy 
 $\displaystyle g(t) \le h(t)+ \int_0^t r(s) g(s) \ds  \text{ for all } t \in (0,T). $ Then,
\begin{equation*}
   \displaystyle  g(t) \le h(t)+ \int_0^t h(s)r(s) e^{\int_s ^{t}{r(\tau )\;\text{d}\tau }}\ds  \quad \text{for all $t \in (0,T)$}.
   \end{equation*}
 \end{lemma}
 \begin{rem}
    {The constants in the stability/error bounds established in this section do not depend exponentially on the final time $T$. {This is achieved by carefully manipulating the coefficients of the last terms in the above displayed inequality in such a way that  $r(\tau)=1/T$.} Then it suffices to  
     apply Lemma~\ref{P1 gronwall} 
     to obtain 
     $e^{\int_s ^{t}{1/T\;\text{d}\tau }}\le e$.}
 \end{rem}
\begin{lemma}[Stability] 
For any $t>0$, the solution  $u_{h}(\cdot, t)$ to \eqref{P1 semi_formulation} with $u_h(0) ={\mathcal R}_h u^0 $ and $u_{ht}(0)= {\mathcal R}_h v^0 $  satisfies 
{\begin{equation*}
    \norm{u_{ht}(t)}+ \norm{u_{h}(t)}_h \lesssim 
 \norm{\mathcal{R}_h u^0}_h+\norm{\mathcal{R}_h v^0}+\norm{f}_{L^2(L^2(\Omega))} .
\end{equation*}
The constant absorbed in "$\lesssim$'' depends on $\alpha$, $\beta$ from \eqref{P1 a_h_properties}, and $T$. }
\end{lemma}
\begin{proof}[\textbf{Proof}]
The choice $v_h = u_{ht}$ in \eqref{P1 semi_formulation} followed by Cauchy--Schwarz and Young {inequalities lead} to 
\begin{align*}
   \frac{\mbox{d}}{\mbox{d}t} \norm{u_{ht}}^2+\frac{\mbox{d}}{\mbox{d}t} a_h (u_{h}, u_h) \le 2\: (f, u_{ht}) 
 \le T\norm{f}^2+\frac{1}{T}\norm{u_{ht}}^2.
\end{align*}
{Integrate from $0$ to $t$,} utilize \eqref{P1 a_h_properties} to bound $a_h(u_h(t), u_h(t))$ and $a_h(u_h(0), u_h(0))$, and obtain 
\begin{align*}
 \norm{u_{ht}(t)}^2+{\alpha} \norm{u_{h}(t) }_h^2  \le \norm{{\mathcal R}_h v^0}^2+ {\beta} \norm{{\mathcal R}_h u^0}_h^2 + T\int_0^t\norm{f(s)}^2 \text{d}s+\frac{1}{T}\int_0 ^t\norm{u_{ht} (s)}^2 \ds.
\end{align*} 
 {Lemma \ref{P1 gronwall} and} elementary manipulations conclude the proof.
\end{proof}

\noindent {We} split the semidiscrete error as 
\begin{equation}
    u(t)-u_h(t)=(u(t)-{\cal  R}_hu(t) )+({\cal R}_hu(t)-u_h(t))=:\rho(t)+\theta(t). \label{P1 semi_split}
\end{equation}
Note that since \eqref{P1 weak_form} holds for any $v \in H^2_0(\Omega)$, in particular, it is true for $Qv_h $, with $v_h \in V_h$. The choice  $Qv_h $ as a test function in \eqref{P1 weak_form} reveals that 
\begin{equation}
    ( u_{tt},Qv_h )+a(u,Qv_h )=(f,Qv_h ) \text{ for all } v_h \in V_h. \label{P1 Qsemi_continuous}
\end{equation} 
Let us now subtract \eqref{P1 semi_formulation} from \eqref{P1 Qsemi_continuous} and employ \eqref{P1 ritz_projection} to obtain
\begin{equation*}
    (u_{tt}-u_{htt},v_h)+a_h({\cal R}_hu-u_h,v_h)=(f-u_{tt},(Q-I)v_h ) \text{ for all } v_h \in V_h.
\end{equation*}
In addition, one readily sees that the error split in \eqref{P1 semi_split} leads to
\begin{equation}  
( \theta_{tt},v_h )+a_h (\theta,v_h)=(f-u_{tt},(Q-I)v_h )-(\rho_{tt}
,v_h ) \text{ for all } v_h \in V_h.\label{P1 semi_error_equation}
\end{equation}
From now on, let us denote 
\begin{align}
L_{(f,u)} & := \norm{f}_{L^\infty {(L^2(\Omega))}}
+ \norm{f_t}_{L^2 ({L^2(\Omega))}}
+\norm{u}_{L^\infty ({H^{2+\gamma_0}(\Omega))}}+ \norm{u_{t}}_{L^\infty ({H^{2+\gamma_0}(\Omega))}} \nonumber\\
&\qquad \qquad +\norm{u_{tt}}_{L^\infty {(L^2(\Omega))}}+ \norm{u_{tt}}_{L^2 {(H^{2+\gamma_0}(\Omega))}} +\norm{u_{ttt}}_{L^2 {(L^2(\Omega))}}. \label{L}
\end{align}
\begin{thm}[Error Control]\label{P1 semi_error_estimates}
Let $u \text{ and } u_h$ solve \eqref{P1 weak_form} and \eqref{P1 semi_formulation}, respectively, and let the regularity results in Lemma \ref{P1 regularity_lemma} hold. {Then, for $t>0$, it holds that 
$$\norm{u_t(t) - u_{ht}(t)}+\norm{u(t) - u_h(t)}_h \lesssim h^{\gamma_0} L_{(f,u)}.$$
The constant in ``$\lesssim$'' depends on $\alpha$, $\beta$ from \eqref{P1 a_h_properties}, $T$, and $C_1$, $C_2$ from Lemmas~\ref{bhcompanion_lem}-\ref{P1 ritz_lemma}.}
\end{thm}
\begin{proof}[\textbf{Proof}]
The choice $v_h= \theta_t$ in \eqref{P1 semi_error_equation} leads to
\begin{equation}\label{eq:01}
 \frac{1}{2}  \left(\frac{\mbox{d}}{\mbox{d}t}\norm{\theta_t}^2 + \frac{\mbox{d}}{\mbox{d}t}a_h(\theta,\theta)\right) =(f-u_{tt},(Q-I)\theta_t )-(\rho_{tt}
,\theta_t ).
\end{equation}
An integration from $0$ to $t$ (with respect  {to} $t$) followed by an integration by parts for the first term on the right-hand side of \eqref{eq:01}{, $\norm{\theta_t(0)}=\norm{\theta(0)}_h=0$,} and  \eqref{P1 a_h_properties}, imply that
\begin{align}
 \frac{1}{2} (\norm{\theta_t}^2 + {\alpha} \norm{\theta}_h^2  )
  & \le  (f-u_{tt},(Q-I)\theta )
  - \int_0 ^ t \big(f_t-u_{ttt},(Q-I)\theta \big) \ds
   -\int_0^t (\rho_{tt}
,\theta_t ) \ds . \label{P1 IBP}
\end{align}
{A Cauchy--Schwarz inequality and \eqref{P1 norm_Q} (with $s=0$, $v=0$) control  the first and second terms on the right-hand side of \eqref{P1 IBP}}
\begin{align*}
     (f-u_{tt},(Q-I)\theta)- \int_0 ^ t \left(f_t-u_{ttt},(Q-I)\theta \right) \ds 
    & \le \norm{f-u_{tt}}\norm{(Q-I)\theta}+\int_0^t  \norm{f_t-u_{ttt}}\norm{(Q-I)\theta}\ds  \\
        & \le  C_1{h^2} \left( \norm{f-u_{tt}} \norm{\theta}_h+\int_0^t\norm{f_t-u_{ttt}} \norm{\theta}_h\ds  \right).
\end{align*}
{Young's} inequality (applied twice)  with $a:=C_1h^2\norm{f-u_{tt}}$ (resp. $a:=C_1h^2\norm{f_t-u_{ttt}}$), $b=\norm{\theta}_h$, and  $\epsilon=2/ \alpha$ ( resp. $\epsilon=2T/ \alpha$, with $\alpha$ from \eqref{P1 a_h_properties}) for the first (resp. second) terms of the right-hand side of the above expression followed by {Cauchy--Schwarz and Young inequalities} with  $a:=\norm{\rho_{tt}}$, $b=\norm{\theta_t}$, and $\epsilon=T$ for the third term on the right-hand side of \eqref{P1 IBP} show
\begin{align*}
 (f-u_{tt},(Q-I)\theta)  &\le
{C_1^2h^4}{\alpha^{-1}} \norm{f-u_{tt}}^2 + \frac {\alpha}{4}\norm{\theta}_h^2, \\
 \int_0 ^ t \left(f_t-u_{ttt},(Q-I)\theta \right) \ds 
    & \le {TC_1^2h^4}{\alpha^{-1}} \int_0^t \norm{f_t-u_{ttt}}^2\ds  
+\frac {\alpha}{4T} \int_0^t \norm{\theta}_h^2 \ds, \\
 \int_0^t(\rho_{tt}
,\theta_t ) \ds & \le  \frac{1}{2} \int_0^t \left(T\norm{\rho_{tt}}^2 +\frac{1}{T}\norm{\theta_t}^2 \right)\ds . 
\end{align*}
{A substitution of this in \eqref{P1 IBP},  triangle inequality (twice), and  the estimate for $\norm{\rho_{tt}}$ from \eqref{P1 norm_ritz} yields}
\begin{align*}
 \nonumber   \norm{\theta_t}^2 + \frac{\alpha}{2} \norm{\theta}_h^2  &\le
    C h^4 \left( \norm{f}^2 + \norm{u_{tt}}^2+\int_0^t \norm{f_t}^2 \ds + \int_0^t \norm{u_{ttt}}^2 \ds \right)\\
& \qquad+{TC^2_2}h^{4\gamma_0}\int_0 ^t \norm{u_{tt}}^2_{H^{2+ \gamma_0}(\Omega)} \ds +\frac{1}{T} \int_0^t \left( \frac{\alpha}{2} \norm{\theta}_h^2  +\norm{\theta_t}^2 \right) \ds.
\end{align*}
{The constant  $C$ depends on $ \alpha$, $C_1$ and is linear in $T$.}\\
{As a consequence of Lemma~\ref{P1 gronwall}, we have}
\begin{align}
   \nonumber  \norm{\theta_t}^2 + \frac{\alpha}{2} \norm{\theta}_h^2 & \lesssim
     h^4 \left( \norm{f}_{L^\infty(L^2(\Omega))}^2 + \norm{u_{tt}}^2_{L^\infty(L^2(\Omega))}+\norm{f_t}^2_{L^2(L^2(\Omega))}+\norm{u_{ttt}}^2_{L^2(L^2(\Omega))} \right)\\
     & \quad+ h^{4\gamma_0} \norm{u_{tt}}^2_{L^2(H^{2+ \gamma_0}(\Omega))}. \label{gr}
\end{align}
{A triangle inequality and Lemma~\ref{P1 ritz_lemma} conclude the proof.}
\end{proof}
\noindent
The next theorem establishes a suitable $L^2-$error bound $\norm{u(t)-u_h(t)}$ under the same regularity as in Theorem~\ref{P1 semi_error_estimates}. {The proof extends the method from \cite{baker1976error}, originally developed for conforming finite elements applied to second-order wave problems, to nonstandard finite elements for fourth-order problems, incorporating an appropriate elliptic projection for the initial data.}
\begin{thm}[Error Control in  {$L^2$}]
Let $u \text{ and } u_h$ solve \eqref{P1 weak_form} and \eqref{P1 semi_formulation}, respectively, and let the regularity results in Lemma \ref{P1 regularity_lemma} hold. {Then, for $t>0$, it holds that }
\[\norm{u(t) - u_{h}(t)}\lesssim h^{2\gamma_0} L_{(f,u)}, \]
where $L_{(f,u)}$ is defined in \eqref{L}, and the constant in ``$\lesssim$'' depends on $\alpha$, $\beta$ from \eqref{P1 a_h_properties}, T, and $C_1$, $C_2$ from Lemmas~\ref{bhcompanion_lem}-\ref{P1 ritz_lemma}. 
\end{thm}
\begin{proof}[\textbf{Proof}]
First we note that equation \eqref{P1 semi_error_equation} can be expressed as
 \begin{equation}
    \frac{\mbox{d}}{\mbox{d}t}(\theta_t,v_h)-(\theta_t,v_{ht})+a_h(\theta,v_h)=(f-u_{tt},(Q-I)v_h)-\frac{\mbox{d}}{\mbox{d}t}(\rho_t,v_h)+(\rho_t,v_{ht}). \label{l2erroreqn}
 \end{equation}
{Let $0 <\tau \le T $ and define}
 \begin{equation}
    \widehat{\theta}(\cdot,t)=\widehat{\theta}(t):=\int_t^\tau \theta(\cdot,s)\ds \qquad  \text{ for } 0 \le t\le T .\label{t_hat}
    \end{equation} 
 The choice of $v_h=\widehat{\theta}(t)$ in \eqref{l2erroreqn} with the observation $\widehat{\theta}_t(t)=-\theta(t)$ from {\eqref{t_hat} implies}
 \begin{equation*}
     \frac{\mbox{d}}{\mbox{d}t}(\theta_t(t),\widehat{\theta}(t))+ \frac{1}{2}\frac{\mbox{d}}{\mbox{d}t}\bigg(\norm{\theta(t)}^2-a_h(\widehat{\theta}(t),\widehat{\theta}(t))\bigg)=(f(t)-u_{tt}(t),(Q-I)\widehat{\theta}(t))-\frac{\mbox{d}}{\mbox{d}t}(\rho_t(t),\widehat{\theta}(t))-(\rho_t(t),{\theta}(t)).
 \end{equation*}
An integration from $0$ to $\tau$ { with respect to $t$} and the observations $\theta(0)=\theta_t(0)=0$ from \eqref{P1 semi_formulation} and $\widehat{\theta}(\tau)=0$ from \eqref{t_hat}, result in the relation 
\begin{equation*}
\frac{1}{2}\bigg(\norm{\theta(\tau)}^2+a_h(\widehat{\theta}(0),\widehat{\theta}(0))\bigg)=\int_0^\tau (f-u_{tt},(Q-I)\widehat{\theta})\dt+(\rho_t(0),\widehat{\theta}(0))-\int_0^\tau(\rho_t,{\theta})\dt.
\end{equation*}
Since $a_h(\widehat{\theta}(0),\widehat{\theta}(0)) \ge 0$ from \eqref{P1 a_h_properties}, {a} Cauchy--Schwarz inequality and \eqref{P1 norm_Q} (choosing $s=0$ and $v=0$) show that
\begin{align}
    \frac{1}{2}\norm{\theta(\tau)}^2 &\le \int_0^\tau \norm{f-u_{tt}}\norm{(Q-I)\widehat{\theta}}\dt+\norm{\rho_t(0)}\norm{\widehat{\theta}(0)}+\int_0^\tau \norm{\rho_t}\norm{{\theta}}\dt\nonumber \\
    & \le  C_1h^2\int_0^\tau \norm{f-u_{tt}}\norm{\widehat{\theta}}_h\dt+\norm{\rho_t(0)}\norm{\widehat{\theta}(0)}+\int_0^\tau \norm{\rho_t}\norm{{\theta}}\dt. \label{l21}
\end{align}
{Young's inequality and the definition of $\widehat{\theta}(\cdot,t)$ from \eqref{t_hat} lead to}
\begin{align*}
    C_1h^2\int_0^\tau \norm{f-u_{tt}}\norm{\widehat{\theta}(\cdot,t)}_h \dt
    &\le \frac{C_1}{2} \Big( h^4 \int_0^\tau \norm{f-u_{tt}}^2\dt+\int_0^\tau\int_t^\tau \norm{{\theta}(\cdot,s)}_h^2 \ds \dt\Big) \nonumber \\
    &\lesssim h^4 \Big(\norm{f}^2_{L^2(L^2(\Omega))}+\norm{u_{tt}}^2_{L^2(L^2(\Omega))}\Big)+\int_0^\tau\int_t^\tau \norm{{\theta}(\cdot,s)}_h^2 \ds \dt. 
\end{align*}
{On the other hand, \eqref{gr}  asserts that}
\begin{align}
        \int_0^\tau\int_t^\tau \norm{{\theta}(\cdot,s)}_h^2 \ds \dt &\lesssim h^4 \Big(\norm{f}_{L^\infty(L^2(\Omega))}^2+\norm{f_t}_{L^2(L^2(\Omega))}^2+\norm{u_{tt}}_{L^\infty(L^2(\Omega))}^2+\norm{u_{ttt}}_{L^2(L^2(\Omega))}^2\Big)\nonumber\\  &\qquad +h^{4\gamma_0}\norm{u_{tt}}_{L^2(H^{2+\gamma_0}(\Omega))}^2. \label{l24}
\end{align}
{Another application} of Young's inequality reveals
\begin{align*}
    \norm{\rho_t(0)}\norm{\widehat{\theta}(0)} +\int_0^\tau \norm{\rho_t}\norm{{\theta}}\dt &\le T\norm{\rho_t(0)}^2+\frac{1}{4T}\int_0^\tau \norm{{\theta}}^2 \dt + T \int_0^\tau \norm{\rho_t}^2 \dt+\frac{1}{4T}\int_0^\tau\norm{{\theta}}^2\dt \nonumber \\
    &\le T \norm{\rho_t(0)}^2+T \int_0^\tau \norm{\rho_t}^2 \dt+\frac{1}{2T}\int_0^\tau \norm{{\theta}}^2 \dt. 
\end{align*}
A combination of \eqref{l21}-\eqref{l24} with the estimates for $\norm{\rho_t(0)}^2$ (available from {Lemma~\ref{P1 ritz_lemma}) yield}
\begin{align*}
    \norm{{\theta(\tau)}}^2 &\le 2C' h^4 \bigg(\norm{f}_{L^\infty(L^2(\Omega))}^2+\norm{f_t}_{L^2(L^2(\Omega))}^2+\norm{u_{tt}}_{L^\infty(L^2(\Omega))}^2+\norm{u_{ttt}}_{L^2(L^2(\Omega))}^2\bigg)\\
&\qquad +2TC_2^2h^{4\gamma_0}\bigg(\norm{v^0}^2_{H^{2+\gamma_0}(\Omega)}+\norm{u_{tt}}_{L^2(H^{2+\gamma_0}(\Omega))}^2\bigg)+\frac{1}{T}\int_0^\tau \norm{{\theta}}^2 \dt.
\end{align*}
{The constant $C'$ depends on $\alpha$, $C_1$ ,$C_2$ and is linear in $T$. Lemma~\ref{P1 gronwall} and Lemma~\ref{P1 ritz_lemma} conclude the proof.}
\end{proof}

\section{Explicit Leapfrog Scheme} \label{P1 fully_discrete_section}
This section describes an explicit, fully-discrete scheme for \eqref{P1 strong_form}. The stability analysis is carried out in Theorem~\ref{P1 explicit_stability} and the corresponding error estimates are presented in {Sub}section~\ref{Explicit_direct_section}. Two approaches work for the error analysis: Theorem~\ref{P1 explicit_Th2} gives a direct proof and an alternate version that utilizes the semidiscrete error bounds from Theorem~\ref{P1 semi_error_estimates} is discussed in Remark~\ref{sem}. Both approaches lead to quasi-optimal estimates under the same regularity assumptions on the exact solution, the CFL conditions, and quasi-uniformity of the underlying triangulation. Even if the use of either approach is common in the literature, to the best of our knowledge, this article is the first to use both approaches and combine lowest-order FE discretization and explicit/implicit time discretization for biharmonic wave equations. We will present  details for the explicit leapfrog scheme in this section, whereas the implicit Newmark scheme will be addressed in Section~\ref{P1 implicit_direct}.

For a positive integer $N$, consider the partition $ 0=t_0 < t_1<t_2< \cdots<t_N=T$ of the interval $[0,T]$
 with $t_n=nk$, and $k=\frac{T}{N}$ being the time step. Let $U^n=U(t_n)$ denote the approximation of the continuous solution $u$ at time $t_n$.  
For any function $\phi(x,t)$, the following notations are adopted:  
\begin{align*}
& \phi^n  := \phi(x,t_n)= \phi(t_n), \quad \phi^{n+1/2}:= \frac{1}{2}\left(\phi^{n+1}+\phi^n \right),
\quad \phi^{n,1/4} :=\frac{1}{4} \left(\phi^{n+1}+2\phi^n+  \phi^{n-1}\right),\\
& \bar{\partial}_t \phi^{n+1/2} :=\frac{\phi^{n+1}-\phi^{n}}{k} ,\quad 
{\bar{\partial}}^2_t \phi^n := \frac{\phi^{n+1}-2\phi^n+\phi^{n-1}}{k^2} ,\quad 
\delta_t \phi^n := \frac{\phi^{n+1}-\phi^{n-1}}{2k}. 
\end{align*}
{These notations will be used throughout in the rest of this article. For example, by $u^n$ (resp. $u_h^n$), we mean the continuous solution $u(t)$ at $t = t_n$ (resp. the semidiscrete solution $u_h(t)$ at $t = t_n$), and so on.} As the first step towards stating the fully-discrete scheme, we define the initial approximations 
$U^0$ and $U^1$ to  $u^0= u(x,0)$ and $u^1=u(x,t_1)$, respectively. For the former we set 
$U^0 :={\cal R}_h{u^0}$, {whereas for the latter  
we proceed with Taylor series expansions}
$$u_{tt}^1=u_{tt}^0+\int_0^{t_1} u_{ttt}(s)  \ds \quad \text{ and } \quad 
 \frac{u^1-u^0}{k}-v^0=\frac{1}{2}ku_{tt}^0+\frac{1}{2k}\int_0^{t_1} (t_1-s)^2u_{ttt}(s) \ds .$$
These relations together with elementary manipulations give  
 \begin{align}
u_{tt}^{1/2}=u_{tt}^0+\frac{1}{2}\int_0^{t_1} u_{ttt}(s)  \ds & ={2k^{-1}\left(\frac{u^1-u^0}{k}-v^0\right)-\frac{1}{k^2}\int_0^{t_1}(t_1-s)^2u_{ttt}(s) \ds +\frac{1}{2}\int_0^{t_1} u_{ttt}(s)  \ds}  \nonumber \\
& = 2k^{-1}(\bar{\partial}_t u^{1/2}-v^0)
-\frac{1}{k^2}\int_0^{t_1}(t_1-s)^2u_{ttt}(s) \ds +\frac{1}{2}\int_0^{t_1} u_{ttt}(s)  \ds . \label{trun}
 \end{align}
{Let us consider}  \eqref{P1 weak_form} at $t=t_0$, $t=t_1$ and add the resulting equations, to arrive at
\begin{equation}
    (u^{1/2}_{tt},v) +a(u^{1/2},v)=(f^{1/2},v) \text{ for all }v \in H_0^2(\Omega). \label{P1 utt_half}
\end{equation}
 Let $2k^{-1} (\bar{\partial}_tU^{1/2}-v^0)$ approximate $u^{1/2}_{tt}$ in \eqref{P1 utt_half} with the truncation error $-\frac{1}{k^2}\int_0^{t_1} (t_1-s)^2u_{ttt}(s) \ds +\frac{1}{2}\int_0^{t_1} u_{ttt}(s)  \ds $ as seen from \eqref{trun}. An approximation $U^1$ for $u(x,t_1)$ {is} then obtained from  
\begin{align}
 2k^{-1}(\bar{\partial}_t U^{1/2} , v_h )+a_h(U^{1/2},v_h) = (f^{1/2}+ {2k^{-1}v^0},v_h )\text{ for all }v_h \in V_h.\label{P1 explicit_ic_1} 
\end{align}
{Given $U^0:={\cal R}_h{u^0} $} and $U^1$ computed from \eqref{P1 explicit_ic_1}, for $n=1, \ldots N-1$, the explicit fully-discrete {problem seeks $U^{n+1} \in V_h$} such that (see \cite{MR4784929,MR4412337})
\begin{equation}
    (\bar{\partial}_t ^2 U^n,v_h) + a_h (U^{n},v_h)=(f^{n},v_h) \text{ for all } v_h \in V_h. \label{P1 explicit_fully_discrete}
\end{equation}
\begin{lemma}[Truncation Errors for the Initial Approximation {(see  {equation} (4.21) in \cite{MR3003381})}] \label{R0}  
Let $2k^{-1}(\bar{\partial}_tU^{1/2}-v^0)$ (resp. $2k^{-1}(\bar{\partial}_tU^{1/2}-u_{ht}^0)$) approximate  $u^{1/2}_{tt}$ (resp. $u_{htt}^{1/2}$). {Then the initial truncation error in this approximation given by }
\begin{align*}
 { R^0 }& := 2k^{-1}(\bar{\partial}_t u^{1/2}-v^0) -u_{tt}^{1/2} = \int_{0}^{t_1}\left(k^{-2}{(t_1 -s)^2}+1/2 \right)u_{ttt}(s)\;{\rm\ds } \\ \bigg( resp. \; \widetilde{R}^0 & := 2k^{-1}(\bar{\partial}_t u_h^{1/2}-u_{ht}^0) -u_{htt}^{1/2}= \int_{0}^{t_1}\left(k^{-2}{(t_1 -s)^2}+{1}/{2}\right)u_{httt}(s)\;{\rm\ds } \bigg) 
\end{align*}
is bounded as 
$  \norm{{R}^0}^2 \le \frac{9}{4} k^2 \norm{u_{ttt}}^2_{L^\infty(L^2(\Omega))}  \; \Big(\text{resp. } \norm{ {\widetilde{R} }^0  }^2 \le \frac{9}{4}{k^2}\norm{u_{httt}}^2_{L^\infty(L^2(\Omega))} \Big). $
\end{lemma}
\begin{lemma}[Truncation Errors for the Explicit Scheme {(see page~7 in \cite{MR3003381})}] \label{tau}
Let $\bar{\partial}^2_t U^n $ denote the approximation of  $u^{n}_{tt}$ (resp. $u_{htt}^{n}$).  
Then the truncation error
\begin{align*}
     \tau^n &:= \bar{\partial}^2_t u^n -u_{tt}^n = \frac{1}{6 }\int_{-k}^{k }k^{-2}\left(k -|s|\right)^3 u_{tttt}(t_n+s){\ds }\\
    \bigg(\text{resp.}\quad  \displaystyle\widetilde{\tau}^n&:= \bar{\partial}^2_t u_h^n -u_{htt}^n = \frac{1}{6}\int_{-k}^{k }k^{-2}\left(k -|s|\right)^3 u_{htttt}(t_n+s){\ds } \bigg) 
    \end{align*}
is bounded as 
$\displaystyle \norm{{ {\tau}}^n }^2 \le \frac{ 1}{126}k^3\int_{t_{n-1}}^{t_{n+1}} \norm{u_{tttt}(s)}^2 {\ds} \; \Big(\text{resp. } \norm{ {\widetilde{\tau} }^n  }^2 \le  \frac{ 1}{126}k^3\int_{t_{n-1}}^{t_{n+1}} \norm{u_{htttt}(s)}^2 {\ds} \Big).$
\end{lemma}
\begin{lemma}[Discrete Inverse Inequality] \label{dinv}
Let $\cal T$ be a quasi-uniform triangulation {of$\; \overline{\Omega}$}. Any  $v_h \in V_h \subset H^2(\cal T)$ satisfies
   $ \norm{v_h}_h \le C_{\rm{inv}} h^{-2} \norm{v_h}$.
\begin{proof}[\textbf{Proof}]
Note that the inequality  
    $$\norm{v_{{h}}}_h \lesssim \sum_{m=0} ^2h^{m-2}_{\cal T}|v_h|_{H^m(\cal T)}$$ 
     holds for  for all $ v_{h} \in V_h$ (see \cite[Theorem 4.1]{MR4376276}).
This bound and the inverse estimates
$$|v_h|_{H^m(K)}\lesssim h^{-m}\norm{v_h}_{L^2(K)},\;  m=0,1,2,$$ valid for quasi-uniform meshes \cite[Lemma 4.5.3]{scot}, concludes the proof. 
\end{proof}
\end{lemma}
\noindent
This section and the rest of the paper uses the discrete {Gronwall} Lemma as stated below.
\begin{lemma}[Discrete {Gronwall} Lemma \cite{MR3003381}]\label{P1 d-gronwall}
 Let $\{a_n\}$, $\{b_n\}$, and $\{c_n\}$ be three non-negative sequences, with $\{c_n\}$ monotone, that satisfy 
  $$\displaystyle    a_n+b_n \le c_n + \mu \sum_{m=0}^{n-1} a_m, \quad \mu >0, \ a_0+b_0 \le c_0 .$$
  Then for $n \ge 0,$ it holds that 
  $    a_n+b_n \le c_n e^{n \mu}.$
\end{lemma}
\begin{rem}
  The constant in the stability/error bounds established in this and the next section does not depend exponentially on the final time \( T \). This is achieved by manipulating the coefficients of the {terms on the right-hand side such that, when the discrete Gronwall inequality is applied, \( c_n \) depends linearly on \( T \), and \( \mu = \frac{k}{T} \). Hence after application of discrete Gronwall,  \( e^{\mu} \leq e \) does not depend on \( T \), which ensures that the constant in the inequality depends linearly on \( T \).}
\end{rem}
\subsection{Stability Analysis}
\begin{thm}[Stability]\label{P1 explicit_stability}
Under the CFL condition $k \le {\beta^{-1/2}C^{-1}_{\rm{inv}}} h^2$ { and the assumption that the {triangulation ${\mathcal T}$} is quasi-uniform}, the scheme \eqref{P1 explicit_ic_1}-\eqref{P1 explicit_fully_discrete} is stable. Moreover, for $1 \le m \le N-1$, the following bound holds:
\begin{align*}
\norm{\bar{\partial}_t U^{m+1/2}}+\norm{U^{m+1/2}}_h \lesssim  \norm{\bar{\partial}_t U^{1/2}}+\norm{U^{1/2}}_h+ \norm{f}_{L^\infty(L^2(\Omega))}.
\end{align*}
{The constant absorbed in "$\lesssim$'' above depends on $\alpha,\,\beta\;$ from \eqref{P1 a_h_properties}, and $T$.}
\end{thm}
\begin{proof}[\textbf{Proof}]
Choose $v_h= 2 k \delta_t U^n$ as a test function in \eqref{P1 explicit_fully_discrete} to obtain
\begin{equation}
    (\bar{\partial}_t ^2 U^n,2 k \delta_t U^n) + a_h (U^{n},2 k \delta_t U^n)=(f^{n},2 k \delta_t U^n) . \label{4.1 all}
    \end{equation}
 {The identities $\bar{\partial}_t ^2 U^n=k^{-1}(\bar{\partial}_tU^{n+1/2}-\bar{\partial}_tU^{n-1/2})$ and $2 k \delta_t U^n=k(\bar{\partial}_tU^{n+1/2}+\bar{\partial}_tU^{n-1/2})$ lead to}  
\begin{equation}
 (\bar{\partial}_t ^2 U^n,2 k \delta_t U^n)= (\bar{\partial}_tU^{n+1/2}-\bar{\partial}_tU^{n-1/2},\bar{\partial}_tU^{n+1/2}+\bar{\partial}_tU^{n-1/2})=\norm{\bar{\partial}_t U^{n+1/2}}^2-\norm{\bar{\partial}_t U^{n-1/2}}^2.  \label{stabilty_ist}
\end{equation}
{Elementary manipulations show} 
\begin{align*}U^n &= \frac{1}{2}\left(U^{n+1/2}+U^{n-1/2}\right)-\frac{1}{4}k(\bar{\partial}_tU^{n+1/2}-\bar{\partial}_tU^{n-1/2}), \\
2 k \delta_t U^n&=2(U^{n+1/2}-U^{n-1/2})=k(\bar{\partial}_tU^{n+1/2}+\bar{\partial}_tU^{n-1/2}).\end{align*} 
{This allows to derive the following relations}
\begin{subequations}
\begin{align}
    a_h (U^{n},2 k \delta_t U^n)&=a_h(U^{n+1/2}+U^{n-1/2},U^{n+1/2}-U^{n-1/2})
    \nonumber\\
    &\quad -\frac{1}{4}k^2a_h(\bar{\partial}_tU^{n+1/2}-\bar{\partial}_t U^{n-1/2},\bar{\partial}_tU^{n+1/2}+\bar{\partial}_t U^{n-1/2})\nonumber\\
    &=a_h(U^{n+1/2},U^{n+1/2})-a_h(U^{n-1/2},U^{n-1/2}) \nonumber \\ 
    & \quad - \frac{1}{4}k^2  a_h(\bar{\partial}_t U^{n+1/2},\bar{\partial}_tU^{n+1/2})
    +\frac{1}{4}k^2 a_h(\bar{\partial}_t U^{n-1/2},\bar{\partial}_tU^{n-1/2}), \label{stability_second} \\
    (f^n,2 k \delta_t U^n)& =k(f^n,\bar{\partial}_tU^{n+1/2}+\bar{\partial}_tU^{n-1/2}). \label{stability_third}
\end{align}
\end{subequations}
{A straightforward} combination of \eqref{stabilty_ist}, \eqref{stability_second}-\eqref{stability_third} 
in \eqref{4.1 all} and summation from $n=1$ to {$n=m$ with $1 \le m \le N-1$ result} in 
\begin{align*}
    \norm{\bar{\partial}_t U^{m+1/2}}^2& +a_h({U^{m+1/2}},{U^{m+1/2}})- \frac{1}{4}k^2    a_h({\bar{\partial}_tU^{m+1/2}},\bar{\partial}_tU^{m+1/2})
    +\frac{1}{4}k^2 a_h(\bar{\partial}_tU^{1/2},\bar{\partial}_tU^{1/2}) \\
    & \quad =  \norm{\bar{\partial}_t U^{1/2}}^2+a_h(U^{1/2},U^{1/2})+ k\sum_{n=1}^{m}  ( f^n,\bar{\partial}_tU^{n+1/2} +\bar{\partial}_tU^{n-1/2}).
\end{align*}
{Observe that  $ \frac{1}{4}k^2 a_h \left(\bar{\partial}_tU^{1/2},\bar{\partial}_tU^{1/2}\right)$ on the left-hand side is non-negative and so \eqref{P1 a_h_properties} leads to}
\begin{align}
 \norm{\bar{\partial}_t U^{m+1/2}}^2+\alpha \norm{{U^{m+1/2}}}_h^2&- \frac{\beta}{4} k^2   \norm{\bar{\partial}_tU^{m+1/2}}^2_h \nonumber\\
    &\le \norm{\bar{\partial}_t U^{1/2}}^2+\beta \norm{U^{1/2}}_h^2+ k\sum_{n=1}^{m}  ( f^n,\bar{\partial}_tU^{n+1/2} +\bar{\partial}_tU^{n-1/2}).\label{explcit stabilty eqn}   
\end{align}
Lemma~\ref{dinv} shows that $\norm{\bar{\partial}_tU^{m+1/2}}_h \le 
 {C_{\text{inv}}h^{-2}} \norm{\bar{\partial}_tU^{m+1/2}} 
 $. This estimate  and the CFL condition $k \le {C^{-1}_{\rm inv}} \beta^{-1/2}h^2$ {control  the left-hand side of \eqref{explcit stabilty eqn} as}
\begin{align}
  \norm{\bar{\partial}_t U^{m+1/2}}^2+\alpha \norm{{U^{m+1/2}}}_h^2&-\frac{\beta}{4} k^2 \norm{\bar{\partial}_tU^{m+1/2}}^2_h 
  \ge \frac{3}{4} \norm{\bar{\partial}_t U^{m+1/2}}^2+\alpha\norm{U^{m+1/2}}^2_h. \label{stab ist}
\end{align}
{A consequence of  Cauchy--Schwarz and Young inequalities} (with $a= \norm{f^n}$, $b=\norm{\bar{\partial}_tU^{n+1/2} +\bar{\partial}_tU^{n-1/2}}$, and $\epsilon=2T$) {reads}
\begin{align*}
 ( f^n,\bar{\partial}_tU^{n+1/2} +\bar{\partial}_tU^{n-1/2}) & \le  \;\norm{f^n} \norm{\bar{\partial}_tU^{n+1/2} +\bar{\partial}_tU^{n-1/2}} 
 \le \;{T} \norm{f^n}^2 +\frac{1}{4T}\norm{\bar{\partial}_tU^{n+1/2} +\bar{\partial}_tU^{n-1/2}}^2 \nonumber\\
 & \le\; {T}\norm{f^n}^2+\frac{1}{2T} \left( \norm{\bar{\partial}_tU^{n+1/2}}^2+\norm{\bar{\partial}_tU^{n-1/2}}^2 \right). 
\end{align*}
On the other hand, we note that 
\begin{align*}
k \sum_{n=1}^{m}\norm{f^n}^2 & \le m k  \norm{f}^2_{L^\infty(L^2(\Omega))}  \le T \norm{f}^2_{L^\infty(L^2(\Omega))}, \\
 \sum_{n=1}^{m}\left(\norm{\bar{\partial}_tU^{n+1/2}}^2+\norm{\bar{\partial}_tU^{n-1/2}}^2\right)& =\norm{\bar{\partial}_tU^{1/2}}^2+\norm{\bar{\partial}_tU^{m+1/2}}^2+2\sum_{n=1}^{m-1}\norm{\bar{\partial}_tU^{n+1/2}}^2.\end{align*}
A combination of {the above} relations with the bound $\frac{k}{2T} \le \frac{1}{2}$ ({used twice}) in \eqref{explcit stabilty eqn} {provides}
\begin{align*}
   \norm{\bar{\partial}_t U^{m+1/2}}^2+\norm{U^{m+1/2}}^2_h \lesssim \norm{\bar{\partial}_t U^{1/2}}^2+ \norm{U^{1/2}}^2_h+ \norm{f}^2_{L^\infty(L^2(\Omega))}+ k/T \sum_{n=1}^{m-1}\norm{\bar{\partial}_t U^{m+1/2}}^2.
\end{align*}
A discrete {Gronwall} inequality from Lemma \ref{P1 d-gronwall} concludes the proof.
\end{proof}

\subsection{Error Estimates for the Explicit Scheme}\label{Explicit_direct_section}
{This subsection discusses a direct approach for error estimates, addressing the derivation from the continuous weak formulation, construction of the explicit scheme, and the definition of the Ritz projection.  Remark~\ref{sem}  discusses an alternative approach of proving error bounds using the semidiscrete scheme. The proof is provided in {Sub}section~\ref{p1 explicit_fully_semi_section}.}

\medskip
\noindent First, let us split the error as
\begin{equation*}
    u(t_n)-U^n=\left(u(t_n)-{\cal R}_hu(t_n)\right)+\left({\cal R}_h u(t_n)-U^n\right):= \rho^n + \zeta^n.
\end{equation*}
The estimates for $\rho^n$ are known from Lemma~\ref{P1 ritz_lemma}, thus it suffices to obtain the error bounds for $ \zeta^n$.

\medskip Since \eqref{P1 weak_form} is true for all $v \in H_0^2(\Omega)$, we can choose in particular $Qv_h \in H_0^2(\Omega)$ as test function, to obtain
\begin{equation}
(u_{tt}(t),Qv_h )+a(u(t),Qv_h)=(f(t),Q v_h) \quad \text{ for all } v_h \in V_h \text{ and } t > 0. \label{P1 explicit_Qun}
\end{equation}
Let us now subtract \eqref{P1 explicit_fully_discrete} and \eqref{P1 explicit_Qun} at $t=t_n$, and this yields the following error equation:
\begin{equation}
( \bar{\partial}_t ^2\zeta^n,v_h ) +a_h ( \zeta^{n},v_h  ) =  ( f^{n}- u_{tt}^n, (Q-I)v_h )+({\tau}^n-{\bar{\partial}_t}^2 \rho^n, v_h ) \quad \text{ for all } v_h \in V_h,
\label{P1 explicit_full_error_eqn}
\end{equation}
where the truncation error $\displaystyle{ { {\tau}}^n }:= \bar{\partial}^2_t u^n -u_{tt}^n$ is defined as in Lemma~\ref{tau}.   

The next lemma provides a bound on the initial error. This result plays a crucial role in the proofs of error estimates in this section and the next. {It should be noted that the proof of Lemma~\ref{P1 lemma_on_norm_initial} does not {require $\mathcal{T}$} to be quasi-uniform; it is true for  
shape-regular triangulations.}  
\begin{lemma}[Initial Error Bounds] \label{P1 lemma_on_norm_initial}
Let the regularity results in Lemma \ref{P1 regularity_lemma} hold true. Then, 
the initial error 
$$\zeta^{1/2}:=(\zeta^{0}+\zeta^{1})/2$$
satisfies
\begin{align*}
    \norm{\bar{\partial}_t \zeta^{1/2}}^2+\norm{ \zeta^{1/2}}_h^2
    \lesssim h^{4\gamma_0}\left(\norm{f}^2_ {L^ {\infty}(L^2(\Omega)) }
    +\norm{u_{tt}}^2_ {L^ {\infty}(L^2(\Omega)) } +\norm{u_{t}}^2_{L^\infty(H^{2+\gamma_0}(\Omega))}\right) +{k^4}\norm{u_{ttt}}^2_{L^ {\infty}(L^2(\Omega)) },
\end{align*}
where the constant absorbed in ``$\lesssim$'' depends on $\alpha$ from \eqref{P1 a_h_properties} and $C_1$, $C_2$ from Lemmas~\ref{bhcompanion_lem}-\ref{P1 ritz_lemma}.
\end{lemma}
\begin{proof}[\textbf{Proof}]
For any $v_h \in V_h$, the properties \eqref{P1 explicit_ic_1} and \eqref{P1 ritz_projection} show that 
\begin{align*}
 2k^{-1} (\bar{\partial}_t \zeta^{1/2},v_h )+ a_h( \zeta^{1/2},v_h)&=2k^{-1}(\bar{\partial}_t R_h u^{1/2} - \bar{\partial}_t U^{1/2},v_h) +a_h( R_hu^{1/2}-U^{1/2},v_h) \nonumber\\
 &= 2k^{-1}(\bar{\partial}_t R_h u^{1/2},v_h)+a(u^{1/2},Qv_h)-(f^{1/2} +2k^{-1}v^0,v_h ) .\nonumber
\end{align*}
The definition of the continuous formulation in \eqref{P1 weak_form} and  $R^0=2k^{-1}(\bar{\partial}_t u^{1/2}-v^0) -u_{tt}^{1/2}$ from Lemma~\ref{R0}  {applied to the right-hand side of the above expression implies that} 
 \begin{equation*}
  2k^{-1} (\bar{\partial}_t \zeta^{1/2},v_h )+ a_h( \zeta^{1/2},v_h)=
 ( f^{1/2}- u_{tt}^{1/2},(Q-1)v_h)-2k^{-1}(\bar{\partial}_t \rho^{1/2},v_h )+(R^0,v_h). 
  \end{equation*}
Since $U^0 ={\cal R}_h{u^0}$, we have that ${\zeta^0}=0$, and so $ 2k^{-1}\zeta^{1/2}=k^{-1}{\zeta^1}={\bar{\partial}_t}\zeta^{1/2}$. {Choose  $v_h= \zeta^{1/2}$ in the last displayed equality}, appeal to \eqref{P1 a_h_properties},  and use the Cauchy--Schwarz inequality to arrive at
\begin{align}
\norm{\bar{\partial} _t \zeta^{1/2}}^2+
{\alpha}\norm{\zeta^{1/2}}_h^2 \nonumber
 &\le
( f^{1/2}- u_{tt}^{1/2},(Q-1)\zeta^{1/2})-(\bar{\partial}_t \rho^{1/2},\bar{\partial} _t \zeta^{1/2})+\frac{1}{2}k(R^0,\bar{\partial} _t \zeta^{1/2}) \\ 
 &\le
\norm{f^{1/2}-u_{tt}^{1/2}}\norm{(Q-1)\zeta^{1/2}}+\norm{\bar{\partial}_t \rho^{1/2}}\norm{\bar{\partial}_t \zeta^{1/2}}+\frac{1}{2}k \norm{R^0}\norm{\bar{\partial}_t \zeta^{1/2}}.\label{initial}
\end{align}
On the other hand, by virtue of \eqref{P1 norm_Q} (with $s=0$ and $v=0$), {we obtain}
\begin{align*}
\norm{f^{1/2}-u_{tt}^{1/2}}\norm{(Q-1)\zeta^{1/2}}
& \le C_1h^2\norm{f^{1/2}-u_{tt}^{1/2}}\norm{\zeta^{1/2}}_h \\
& \le {{C_1^2}{{ \alpha}^
{-1}}}h^4\left( \norm{f^{1/2}}^2+\norm{u_{tt}^{1/2}}^2 \right)+{\alpha}/{4}\norm{\zeta^{1/2}}^2_h,
\end{align*}
where  Young's inequality  (with $a=C_1h^2\norm{f^{1/2}-u_{tt}^{1/2}}$, $b=\norm{\zeta^{1/2}}_h$ and $\epsilon=2/\alpha$) plus elementary manipulations have been employed in the last step.\\
{Young's inequality} applied to second (resp. third) term of \eqref{initial} with $a=\norm{\bar{\partial}_t \rho^{1/2}}$ (resp. $a=k \norm{R^0}$), $b=\norm{\bar{\partial}_t \zeta^{1/2}}$ and $\epsilon=2$ (resp. $\epsilon=1$) lead to the estimate
\begin{equation*}
 \norm{\bar{\partial}_t \rho^{1/2}}\norm{\bar{\partial}_t \zeta^{1/2}}+\frac{1}{2}k \norm{R^0}\norm{\bar{\partial}_t \zeta^{1/2}} \le \norm{\bar{\partial}_t \rho^{1/2}}^2+\frac{1}{4}k^2\norm{R^0}^2 
+\frac{1}{2}\norm{\bar{\partial}_t \zeta^{1/2}}^2. 
\end{equation*}
In addition, the bound for $\rho_t$ from Lemma~\ref{P1 ritz_lemma} reveals
\begin{align*}     
\norm{\bar{\partial}_t \rho^{1/2}}  =\norm{\int_0^{t_1}{k^{-1}}{\rho_t(t)\;dt}} \le C_2 h^{2\gamma_0}\norm{u_{t}}_{L^\infty(H^{2+\gamma_0}(\Omega))}. 
\end{align*}
Finally, the proof follows from a combination of all these steps in \eqref{initial}, plus the bounds for $R^0$ available from Lemma~\ref{R0}.  
\end{proof}
\begin{thm}[Error Estimate]\label{P1 explicit_Th2} Let $\cal T$ be a quasi-uniform triangulation of $\Omega$ and $u$ ({resp. $U^{n+1}$}) solve \eqref{P1 weak_form} (resp. \eqref{P1 explicit_fully_discrete}). Then, under the assumptions of Lemma \ref{P1 regularity_lemma} and the CFL condition $k\le \beta^{-1/2} {C^{-1}_{\rm inv}} h^2$, for $1 \le m\le N-1$,  we have 
\begin{align*} \norm{\bar{\partial}_t u^{m+1/2}-\bar{\partial}_t U^{m+1/2} } + \norm{ u^{m+1/2}- U^{m+1/2} }_h \lesssim  h^{\gamma_0} L_{(f,u)} +k^2 M_{(u)},
\end{align*}
 {where $L_{(f,u)}$ is given by \eqref{L},}
$$M_{(u)}:=  \norm{u_{ttt}}_{L^\infty {(L^2(\Omega))}} +\norm{u_{tttt}}_{L^2{(L^2(\Omega))}},$$ 
and the  constant absorbed in ``$\lesssim$'' depends on $\alpha$, $\beta$ from \eqref{P1 a_h_properties},  $T$,  and $C_1,C_2$ from Lemma~\ref{bhcompanion_lem}-\ref{P1 ritz_lemma}.
\end{thm}
\begin{proof}[\textbf{Proof}]
First, it is easy to see that the choice $v_h = 2 k \delta_t \zeta^n$ in \eqref{P1 explicit_full_error_eqn} leads to
\begin{equation*}
    (\bar{\partial}_t ^2 \zeta^n,2 k \delta_t \zeta^n) + a_h (\zeta^{n},2 k \delta_t \zeta^n)= ( f^{n}- u_{tt}^n, (Q-I)2 k \delta_t \zeta^n)+({\tau}^n-{\bar{\partial}_t}^2 \rho^n, 2 k \delta_t \zeta^n ) .
    \end{equation*}
  {We proceed} with the arguments in Theorem~\ref{P1 explicit_stability} (see \eqref{4.1 all}--\eqref{stab ist}) to obtain
\begin{align}
  \frac{3}{4} \norm{\partial _t \zeta^{m+1/2}}^2+\alpha \norm{ \zeta^{m+1/2}}^2_h &\le\norm{\bar{\partial}_t \zeta^{1/2}}^2+\beta \norm{\zeta^{1/2}}^2_h+ 2 \sum_{n=1}^{m} ( f^n-u_{tt}^n,(Q-I) (\zeta^{n+1/2} -\zeta^{n-1/2}) )\nonumber\\
  &\qquad \qquad \qquad \qquad \quad \quad \; +2k \sum_{n=1}^{m}  ({\tau}^n-{\bar{\partial}_t}^2 \rho^n, \delta_t \zeta^{n} ).   \label{P1 explicit_full_equation}
\end{align}
{The identity}
$$\sum_{n=1}^m g_n(h_n-h_{n-1})= g_mh_m-g_0h_0-\sum_{n=1}^{m} (g_n-g_{n-1})h_{n-1}$$ 
{leads to}
\begin{align*}
I_1&:=\sum_{n=1}^{m} ( f^n-u_{tt}^n,(Q-I) (\zeta^{n+1/2} -\zeta^{n-1/2}) )\\
&=( f^{m}- u_{tt}^{m}, (Q-I) \zeta^ {m+1/2} ) 
 -(  f^{0}- u_{tt}^{0}, (Q-I) \zeta^ {1/2} ) + \sum_{n=0}^{m-1} ( f^{n+1} - f^{n} - u_{tt}^{n+1}+u_{tt}^n,
(Q-I)\zeta^{n+1/2} )\\
&\le C_1h^2\bigg( \norm{f^{m}-  u_{tt}^{m} } \norm{ \zeta^ {m+1/2}}_h
+\norm {  f^{0}- u_{tt}^{0}}\norm{\zeta^ {1/2}}_h+ \sum_{n=0}^{m-1} \norm{\int_{t_{n}}^{t_{n+1}}({f_{t}(t)-u_{ttt}(t)})\dt } \norm{ \zeta^{n+1/2}}_h \bigg).
\end{align*}
{The last step utilizes a} Cauchy--Schwarz inequality and  \eqref{P1 norm_Q} (for $s=0,\;v=0$). 
{The Cauchy--Schwarz inequality} 
\[\norm{\int_{t_{n}}^{t_{n+1}}g(t)\dt }\le \sqrt{k}\left(\int_{t_{n}}^{t_{n+1}}\norm{g(t)}^2\dt \right)^{1/2}\] 
{applies} to the last term on the right-hand side of the above expression for {$I_1$. This and Young's inequality provide}
\begin{align}
   I_1 &\le  {C_1^2{\alpha}^{-1}h^4}\bigg( \norm{f^m-u_{tt}^m}^2+\norm{f^0-u_{tt}^0}^2\bigg) +\frac{\alpha}{4}\bigg(\norm{\zeta^{m+1/2}}_h^2 +\norm{\zeta^{1/2}}_h^2 \bigg)\nonumber\\
   & \qquad \quad +{C_1^2{\alpha}^{-1}Th^4}\sum_{n=0}^{m-1}\int_{t_n}^{t_{n+1}}\norm{f_t(t)-u_{ttt}(t)}^2\dt +\frac{\alpha}{4T}k\sum_{n=0}^{m-1}\norm {\zeta^{n+1/2}}^2_h.\nonumber
\end{align}
{The inequalities $\norm{g+h}^2 \le 2(\norm{g}^2+\norm{h}^2)$  and  $\sum_{n=0}^{m-1} \int_{t_{n}}^{t_{n+1}}\norm{q(t)}^2\dt \le\norm{q(t)}^2_{L^2(L^2(\Omega))}$ lead to}
\begin{align}
I_1 &\le  {4C_1^2{\alpha}^{-1}h^4}\bigg( \norm{f}^2_{L^\infty(L^2(\Omega))}+\norm{u_{tt}}^2_{L^\infty(L^2(\Omega))}\bigg) +\frac{\alpha}{4}\bigg(\norm{\zeta^{m+1/2}}_h^2   +\norm{\zeta^{1/2}}_h^2 \bigg)\nonumber\\
   & \quad 
   +{2C_1^2{\alpha}^{-1}Th^4}\bigg(\norm{f_t}^2_{L^2(L^2(\Omega))}+\norm{u_{ttt}}^2_{L^2(L^2(\Omega))}\bigg)
  +\frac{\alpha}{4T}k\sum_{n=0}^{m-1}\norm {\zeta^{n+1/2}}^2_h.\label{I1}
\end{align}
{On the other hand, $2\delta_t\zeta^n = \bar{\partial}_t \zeta^{n+1/2}+\bar{\partial}_t \zeta^{n-1/2}$ and Cauchy--Schwarz inequality reveals that}
\begin{align*}
    I_2 := 2k \sum_{n=1}^m (\tau^n-{\bar{\partial}_t}^2 \rho^n,\delta_t\zeta^n )& = k \sum_{n=1}^m (\tau^n-{\bar{\partial}_t}^2 \rho^n, \bar{\partial}_t \zeta^{n+1/2}+\bar{\partial}_t \zeta^{n-1/2} ) \nonumber \\
    & \le  k \sum_{n=1}^m \norm{\tau^n-{\bar{\partial}_t}^2 \rho^n} \norm{ \bar{\partial}_t \zeta^{n+1/2}+\bar{\partial}_t \zeta^{n-1/2}}.
\end{align*}
{Young's inequality (with $a=\norm{\tau^n -{\bar{\partial}_t^2 \rho^n}}$, $b=\norm{ \bar{\partial}_t \zeta^{n+1/2}+\bar{\partial}_t \zeta^{n-1/2}}$,  $\epsilon=2T$) and   elementary manipulations lead to}
\begin{align}
I_2
 &\le k \sum_{n=1}^m  \left( T\norm{\tau^n-{\bar{\partial}_t}^2 \rho^n}^2 + \frac{1}{4T} \norm{ \bar{\partial}_t \zeta^{n+1/2}+\bar{\partial}_t \zeta^{n-1/2}}^2 \right) \nonumber\\
 &\le k \sum_{n=1}^m  \left(2T\norm{\tau^n}^2 +2T\norm{{\bar{\partial}_t}^2 \rho^n}^2+ \frac{1}{2T} \left( \norm{ \bar{\partial}_t \zeta^{n+1/2}}^2+\norm{\bar{\partial}_t \zeta^{n-1/2}}^2 \right) \right) .\nonumber
\end{align}
{The truncation error $\tau^n$ from Lemma~\ref{tau}, $\norm{\bar{\partial}_t^2 \rho^n}^2  \le \frac{2}{3}k^{-1}\int_{t_{n-1}}^{t_{n+1}} \norm{\rho_{tt}(t)}^2 \dt $ from  Taylor series, and Lemma~\ref{P1 ritz_lemma} lead to}
 \begin{align}
I_2 &\le\frac{T}{63}k^4\sum_{n=1}^{m}\int_{t_{n-1}}^{t_{n+1}} \norm{u_{tttt}(t)}^2 \dt +\frac{4}{3} C_2Th^{4\gamma_0} \sum_{n=1}^{m}\int_{t_{n-1}}^{t_{n+1}}\norm{u_{tt}(t)}_{H^{2+\gamma_0}(\Omega)}^2 \dt \nonumber\\ 
 &\qquad + \frac{k}{2T} \norm{ \bar{\partial}_t \zeta^{1/2}}^2+ \frac{k}{2T} \norm{ \bar{\partial}_t \zeta^{m+1/2}}^2+ \frac{k}{T} \sum_{n=1}^{m-1}\norm{\bar{\partial}_t \zeta^{n+1/2}}^2 . \label{I2}
\end{align}
{A combination of \eqref{I1}-\eqref{I2} in \eqref{P1 explicit_full_equation},  the fact that $ \frac{k}{2} \le \frac{T}{2}$ in the third and fourth terms on the right-hand side of the above expression,  and  Lemma~\ref{P1 lemma_on_norm_initial} yield  the following bound} 
\[\norm{\bar{\partial}_t \zeta^
{m+1/2}}^2+ \norm{\zeta^{m+1/2} }_h^2 \nonumber \lesssim  h^{4\gamma_0}N_{(f,u)}^2
+{k^4} M_{(u)}^2+
\frac{k}{T}  \sum_{n=0}^{m-1} \norm{\zeta^{n+1/2}}^2_h+\frac{k}{T} \sum_{n=1}^{m-1}\norm{\bar{\partial}_t \zeta^{n+1/2}}^2 \]
with 
\begin{equation}\label{eqn:N}
N_{(f,u)}^2:=\norm{f}^2_{L^\infty(L^2(\Omega))}+\norm{f_t}^2_{L^2(L^2(\Omega))} +\norm{u_{t}}^2_{L^\infty(H^{2+\gamma_0}(\Omega))} +\norm{u_{tt}}^2_{L^\infty(L^2(\Omega))}+\norm{u_{tt}}^2_{L^2(H^{2+\gamma_0}(\Omega))}+\norm{u_{ttt}}^2_{L^2(L^2(\Omega))}.
\end{equation}
{Then,  Lemma~\ref{P1 d-gronwall} provides}
\begin{align}
&\norm{\bar{\partial}_t \zeta^
{m+1/2}}^2+ \norm{\zeta^{m+1/2} }_h^2  \lesssim  h^{4\gamma_0}N_{(f,u)}^2
+{k^4} M_{(u)}^2. \label{cor}
\end{align}
{Triangle inequalities lead to}
\[\norm{\bar{\partial}_t u^{m+1/2}-\bar{\partial}_t U^{m+1/2} } + \norm{ u^{m+1/2}- U^{m+1/2} }_h  \le \| \rho^{m+1/2}  \|+ \| \bar{\partial}_t \rho^{m+1/2}  \|+\| \zeta^{m+1/2}  \|  + \| \bar{\partial}_t \zeta^{m+1/2}  \|.\] 
{This, \eqref{cor}, and  Lemma \ref{P1 ritz_lemma} conclude the proof.}
\end{proof}
\begin{cor}[An $L^2-$Norm Estimate]\label{corr}
   Suppose that $u$ and  { $U^{n+1}$ solves} \eqref{P1 weak_form} and  \eqref{P1 explicit_fully_discrete}, respectively. Then under the assumptions of Theorem \ref{P1 explicit_Th2}, for $1\le m\le N-1$, the following error estimate holds
\begin{align*}
  \norm{ u^{m+1}- U^{m+1} } \lesssim  h^{2\gamma_0} L_{(f,u)} +k^2 M_{(u)} .
\end{align*}
\end{cor}
\begin{proof}[\textbf{Proof}]
We consider the bound in {\eqref{cor} and} utilize 
that $\zeta^{m+1}=\zeta^{m+1/2}+\frac{1}{2}k\bar{\partial}_t\zeta^{m+1/2}.$
{Since} $\norm{\zeta^{m+1/2} } \le \norm{\zeta^{m+1/2} }_h$, {we can  conclude} that 
\[\norm{\zeta^{m+1}} \lesssim \norm{\bar{\partial}_t \zeta^
{m+1/2}}+ \norm{\zeta^{m+1/2} }_h\lesssim  h^{2\gamma_0}N_{(f,u)}\nonumber+k^2M_{(u)}.\]
Therefore, a triangle inequality shows $ \norm{ u^{m+1}- U^{m+1} }  \le \| \rho^{m+1}  \| +\| \zeta^{m+1}  \| $. This estimate and Lemma \ref{P1 ritz_lemma} lead to the desired result. 
\end{proof}

\begin{rem} [Error Analysis Using the Semidiscrete Scheme]\label{sem}
The error analysis for \eqref{P1 explicit_fully_discrete} can be done using semidiscrete estimates. The proof is given in the Appendix (see {Sub}section \ref{p1 explicit_fully_semi_section}). However, it is not clear whether the bounds for $\norm{u_{httt}}_{L^2(L^2(\Omega))}$ and $\norm{u_{htttt}}_{L^2(L^2(\Omega))}$ can be established using the regularity assumptions on the {given} data as stated in Lemma~\ref{P1 regularity_lemma}. 
\end{rem}

\section{Implicit Newmark Scheme} \label{P1 implicit_direct}
The stability result in Theorem \ref{P1 explicit_stability} indicates that the explicit schemes 
are {\it conditionally stable}.  
{Moreover, ${\mathcal T}$} is assumed to be {\it quasi-uniform} in Section \ref{P1 fully_discrete_section}. The implicit scheme introduced in this section  motivated by {\cite{MR0349045} for the wave equation {\it circumvents both these assumptions}. 

\medskip
 First, we specify that the initial approximations $U^0$ and $U^1$ are defined as in the explicit scheme:
\begin{align}
U^0 &={\cal R}_h{u^0} \text { and } 
 2k^{-1}(\bar{\partial}_t U^{1/2} , v_h )+a_h(U^{1/2},v_h) = (f^{1/2}+ {2k^{-1}v^0},v_h )  \text{ for all }v_h \in V_h.\label{P1 ic_2 } 
\end{align}
Given $U^0 \text{ and } U^1$, for $n=1,\ldots N-1$, the implicit fully-discrete problem seeks $U^{n+1} \in V_h$ such that (see \cite[Equation (34)]{MR4412337}, with $\beta=1/4$ and $\gamma=1/2$)
\begin{equation}
    (\bar{\partial}_t ^2 U^n,v_h) + a_h (U^{n,1/4},v_h )= (f^{n,1/4},v_h ) \text{ for all } v_h \in V_h. \label{P1 fully_discrete}
    \end{equation}
\begin{lemma}[Truncation Errors for the Implicit Scheme {(see page~11 in \cite{MR3003381})}]\label{r}
Let $\bar{\partial}^2_t U^n $  approximate  $u^{n,1/4}_{tt}$ (resp. $u_{htt}^{n,1/4}$). Then the truncation error  
\begin{align*}
\displaystyle{r^n }:= \bar{\partial}^2_t u^n -u_{tt}^{n,1/4} =\frac {1}{12} \int_{-k}^{k} (k -|s|)(2(1-k^{-1}{|s|})^2-3)u_{tttt}(t_n+s)\ds \\
\bigg(\text{resp.}\quad  \displaystyle \widetilde{r}^n:= \bar{\partial}^2_t u_h^n -u_{htt}^{n,1/4}= \frac{1}{12} \int_{-k}^{k} (k -|s|) (2(1-k^{-1}{|s|})^2-3)u_{htttt}(t_n+s)\ds \bigg)
\end{align*} 
is bounded by 
$\displaystyle   \norm{r^n }^2 \le  \frac{41}{2520}k^3 \int_{t_{n-1}}^{t_{n+1}} \norm{u_{tttt}}^2 {\ds} \quad \bigg(\text{resp. } \norm{{\widetilde{r} }^n  }^2  \le \frac{41}{2520} k^3\int_{t_{n-1}}^{t_{n+1}} \norm{u_{htttt}}^2 {\ds } \bigg). $
\end{lemma}
\subsection{Stability Analysis}
\begin{thm}[Stability]\label{P1 full_stabilty}
The implicit scheme \eqref{P1 ic_2 }-\eqref{P1 fully_discrete} is stable and for $1 \le m \le N-1$
\begin{equation*}
\norm{\bar{\partial}_tU^{m+1/2} }+\norm{U^{m+1/2}}_h \lesssim  \norm{\bar{\partial}_t U^{1/2}} +\norm{U^{1/2}}_h+\norm{f}_{L^\infty(L^2(\Omega))}.
\end{equation*}
{The  constant absorbed in "$\lesssim$'' depends on $\alpha$, $\beta$ from \eqref{P1 a_h_properties} and  $T$.}
\end{thm}
\begin{proof}[\textbf{Proof}]
We choose $v_h=2k \delta_tU^n$ in \eqref{P1 fully_discrete} to obtain
 \begin{equation*}
 (\bar{\partial}_t ^2 U^n,2k \delta_tU^n) + a_h (U^{n,1/4},2k \delta_tU^n)= (f^{n,1/4},2k \delta_tU^n ).    
\end{equation*}
 The identities \begin{align*}
 \bar{\partial}_t ^2 U^n&=\frac{1}{k}(\bar{\partial}_t U^{n+1/2}-\bar{\partial}_t U^{n-1/2}),\\
 U^{n,1/4}&=\frac{1}{2}(U^{n+1/2}+U^{n-1/2}),\\
 2k \delta_t U^n&=2(U^{n+1/2}-U^{n-1/2})=k(\bar{\partial}_t U^{n+1/2}+\bar{\partial}_t U^{n-1/2})
 \end{align*}show 
 \begin{align*}
    ( \bar{\partial}_t U^{n+1/2}-\bar{\partial}_t U^{n-1/2} ,\bar{\partial}_t U^{n+1/2}+\bar{\partial}_t U^{n-1/2} ) &+a_h ( U^{n+1/2}+U^{n-1/2},U^{n+1/2}-U^{n-1/2})\\
    &=
    k \;(f^{n,1/4},\bar{\partial}_tU^{n+1/2}+\bar{\partial}_t U^{n-1/2}).
 \end{align*} 
{A summation from $n=1,\ldots, m$, for any $m$, $1 \le m\le N-1$ leads to}
 \begin{align*}  \norm{\bar{\partial}_t U^{m+1/2}}^2 &+ a_h({U^{m+1/2}},{U^{m+1/2}})=\norm{\bar{\partial}_t U^{1/2}}^2+a_h({U^{1/2}},{U^{1/2}})+
    k\sum_{n=1}^{m}(f^{n,1/4},\bar{\partial}_tU^{n+1/2}+\bar{\partial}_t U^{n-1/2}).
 \end{align*} 
Then, the continuity and ellipticity of $a_h(\cdot,\cdot)$ (cf. \eqref{P1 a_h_properties}) together with Cauchy--Schwarz inequality yield 
 \begin{align*}
    \norm{\bar{\partial}_t U^{m+1/2}}^2+&\alpha \norm{U^{m+1/2}}_h^2 \le \norm{\bar{\partial}_t U^{1/2}}^2+\beta\norm{U^{1/2}}_h^2+
k\sum_{n=1}^{m}\norm{f^{n,1/4}}\norm{\bar{\partial}_tU^{n+1/2}+\bar{\partial}_t U^{n-1/2}}.
 \end{align*} 
 On the other hand, Young's inequality (with $a=\norm{f^{n,1/4}},b=\norm{\bar{\partial}_t U^{n+1/2}+\bar{\partial}_t U^{n-1/2}}$, $\epsilon=2T$) applied to last term on the right-hand side and elementary manipulations imply that 
\begin{align*}
    \norm{\bar{\partial}_t U^{m+1/2}}^2+\alpha \norm{U^{m+1/2}}_h^2 &\le \norm{\bar{\partial}_t U^{1/2}}^2+\beta\norm{U^{1/2}}_h^2+T k \sum_{n=1}^{m}\norm{f^{n,1/4}}^2+\frac{k}{2T}\norm{\bar{\partial}_tU^{1/2}}^2\\   & \quad+\frac{k}{2T}\norm{\bar{\partial}_tU^{m+1/2}}^2+\frac{k}{T}\sum_{n=1}^{m-1}\norm{\bar{\partial}_tU^{n+1/2}}^2.
 \end{align*} 
 Note that $T k \sum_{n=1}^{m}\norm{f^{n,1/4}}^2 \le mT k  \norm{f}^2_{L^\infty(L^2(\Omega))} \le T^2 \norm{f}^2_{L^\infty(L^2(\Omega))} $ and $\frac{k}{2T}\le\frac{1}{2} $. 
 This leads to 
 \begin{align*}
    \norm{\bar{\partial}_t U^{m+1/2}}^2+ \norm{U^{m+1/2}}_h^2 \lesssim 
    \norm{\bar{\partial}_t U^{1/2}}^2 +\norm{U^{1/2}}_h^2+\norm{f}_{L^\infty(L^2(\Omega))}^2+ \frac{k}{T} \sum_{n=1}^{m-1}\norm{\bar{\partial}_t U^{n+1/2}}^2.
    \end{align*}
{An application of Lemma \ref{P1 d-gronwall} concludes the proof.}
\end{proof}

\subsection{Error Estimates}
Since \eqref{P1 weak_form} is true for all $v \in  H^2_0(\Omega)$, we have in particular  
\begin{equation*}   
({u_{tt}(t)},Qv_h )+a(u(t),Qv_h)=(f(t),Q v_h) \text{ for all }v_h \in V_h \text{ and } t > 0. 
\end{equation*}
Let us recall that $r^n:= \bar{\partial}^2_t u^n -u_{tt}^{n,1/4}$  from Lemma~\ref{r}. We then multiply the above equation by $1/4$ at $t=t_{n+1}$,\ $t=t_{n-1}$ and $1/2$ at $t=t_{n}$, and sum up the resulting three equations to obtain 
\begin{equation}
({{\bar{\partial}_t}^2 u^n},Qv_h )+a(u^{n,1/4},Qv_h)=(f^{n,1/4}+{ r}^n,Q v_h) \quad \text{ for all }v_h \in V_h, \quad n=1,\ldots, N-1. \label{P1 Q_continuous_formulation}
\end{equation}
{Subtract \eqref{P1 fully_discrete} from \eqref{P1 Q_continuous_formulation}}  and utilize the definition of the Ritz projection \eqref{P1 ritz_projection} to obtain
\begin{align*}
(f^{n,1/4}, (Q-I)v_h )+(r^n,Qv_h)
    &=(\bar{\partial}_t ^2 u^n, Q v_h)- ( \bar{\partial}_t ^2 U^n, v_h )+a_h(\mathcal{R}_hu^{n,1/4},v_h)  - a_h(U^{n,1/4},v_h)\\ 
   &= ({\bar{\partial}_t}^2 \rho^n, v_h )+ ( \partial _t^2\zeta^n,v_h ) +a_h ( \zeta^{n,1/4},v_h )+(r^n+u_{tt}^{n,1/4},(Q-I)v_h).
\end{align*}
For $n=1, \dots, N-1$, the error equation is given by
\begin{equation}
( \bar{\partial}_t ^2\zeta^n,v_h ) +a_h ( \zeta^{n,1/4},v_h  ) =  ( f^{n,1/4}- u_{tt} ^{n,1/4}, (Q-I)v_h )+(r^n-{\bar{\partial}_t}^2 \rho^n,v_h) \text{ for all }v_h \in V_h .\label{P1 full_error_eqn}
\end{equation}

 \begin{thm}[Error Estimate]
 Let the regularity results in Lemma \ref{P1 regularity_lemma}  hold. {Then, for $1 \le m\le N-1$, we have that}
 \begin{align*} \norm{\bar{\partial}_t u^{m+1/2}-\bar{\partial}_t U^{m+1/2} } + \norm{ u^{m+1/2}- U^{m+1/2} }_h \lesssim  h^{\gamma_0} L_{(f,u)} +k^2 M_{(u)}, 
\end{align*}
where $L_{(f,u)}$, 
$M_{(u)}$ are defined in Theorem \ref{P1 explicit_Th2} and the  constant absorbed in ``$\lesssim$'' depends on $\alpha$, $\beta$ from \eqref{P1 a_h_properties},  $T$,  and $C_1$, $C_2$ from Lemmas~\ref{bhcompanion_lem}-\ref{P1 ritz_lemma}.
    \end{thm}
\begin{proof}[\textbf{Proof}]
The choice $v_h = 2 k \delta_t \zeta^n=2(\zeta^{n+1/2} -\zeta^{n-1/2})$ in \eqref{P1 full_error_eqn} and steps analogous to Theorem~\ref{P1 full_stabilty} show
\begin{align}
\norm{\bar{\partial}_t \zeta^{m+1/2}}^2+& \alpha\norm{{\zeta^{m+1/2}}}_h^2-\norm{\bar{\partial}_t\zeta^{1/2}}^2-\beta \norm{{\zeta^{1/2}}}_h^2\nonumber\\
&=
2\sum_{n=1}^{m}
 ( f^{n,1/4}- u_{tt}^{n,1/4}, (Q-I)(\zeta^{n+1/2} -\zeta^{n-1/2}))+2k\sum_{n=1}^{m}(r^n-{\bar{\partial}_t}^2 \rho^n,\delta _t \zeta^{n} ).\label{J}
\end{align} 
Then the identity $\sum_{n=2}^m g_n(h_n-h_{n-1})= g_mh_m-g_1h_1-\sum_{n=2}^{m} (g_n-g_{n-1})h_{n-1}$ shows that 
\begin{align}
 J_1&:=\sum_{n=1}^{m}
 ( f^{n,1/4}- u_{tt}^{n,1/4}, (Q-I)(\zeta^{n+1/2} -\zeta^{n-1/2}))\nonumber\\
 &=(f^{1,1/4}-u_{tt}^{1,1/4},(Q-I)(\zeta^{3/2}-\zeta^{1/2}  ))+ \sum^m_{n=2} \nonumber  ( f^{n,1/4}- u_{tt}^{n,1/4} ,  (Q-I)( \zeta^ {n+1/2} 
-\zeta^{n-1/2} )) \nonumber \\
 &= ( f^{m,1/4}- u_{tt}^{m,1/4}, (Q-I) \zeta^ {m+1/2}) 
 - (  f^{1,1/4}-u_{tt}^{1,1/4}, (Q-I) \zeta^ {1/2} )  \nonumber \\
&\qquad \qquad+ \sum_{n=1}^{m-1} ( \int_{t_{n-1}}^{t_{n+2}}( f_t(t)-u_{ttt}(t))\dt ,
(Q-I)\zeta^{n+1/2}) . \nonumber
\end{align}
Steps similar to $I_1$ in Theorem~\ref{P1 explicit_Th2} and  $\displaystyle \sum_{n=1}^{m-1} \int_{t_{n-1}}^{t_{n+2}}\norm{q(t)}^2\dt \le3\norm{q}^2_{L^2(L^2(\Omega))}$ lead to
\begin{align}
J_1 &\le  {4C_1^2{\alpha}^{-1}h^4}\bigg( \norm{f}^2_{L^\infty(L^2(\Omega))}+\norm{u_{tt}}^2_{L^\infty(L^2(\Omega))}\bigg) +{6C_1^2{\alpha}^{-1}Th^4}\bigg(\norm{f_t}^2_{L^2(L^2(\Omega))}+\norm{u_{ttt}}^2_{L^2(L^2(\Omega))}\bigg)\nonumber\\
   & \qquad \; +\frac{\alpha}{4}\bigg(\norm{\zeta^{m+1/2}}_h^2 +\norm{\zeta^{1/2}}_h^2 \bigg)+\frac{\alpha}{4T}k\sum_{n=1}^{m-2}\norm {\zeta^{n+1/2}}^2_h.\label{J1}
\end{align}
We define now $J_2:=2k\sum_{n=1}^{m}(r^n-{\bar{\partial}_t}^2 \rho^n,\delta _t \zeta^{n} )$ and follow the similar steps used to bound $I_2$ in Theorem~\ref{P1 explicit_Th2} with the bound for $r^n$ from Lemma~\ref{r} to obtain
\begin{align}
J_2 &\le\frac{41}{630}Tk^4\norm{u_{tttt}}^2_{L^2(L^2(\Omega))}+\frac{8}{3}C_2Th^{4\gamma_0} \norm{u_{tt}}^2_{L^2({H^{2+\gamma_0}(\Omega))}}\nonumber\\ 
 &\qquad + \frac{k}{2T}\norm{ \bar{\partial}_t \zeta^{1/2}}^2+ \frac{k}{2T}\norm{ \bar{\partial}_t \zeta^{m+1/2}}^2+ \frac{k}{T}\sum_{n=1}^{m-1}\norm{\bar{\partial}_t \zeta^{n+1/2}}^2 . \label{J2}
\end{align}
Next, we observe that a combination of \eqref{J1}-\eqref{J2} in \eqref{J} and the bounds from Lemma~\ref{P1 lemma_on_norm_initial} lead to
\begin{align*}
&\norm{\bar{\partial}_t \zeta^
{m+1/2}}^2+ \norm{\zeta^{m+1/2} }_h^2  \lesssim  h^{4\gamma_0}N_{(f,u)}^2
+{k^4} M_{(u)}^2+
\frac{k}{T}  \sum_{n=0}^{m-1} \norm{\zeta^{n+1/2}}^2_h+\frac{k}{T} \sum_{n=1}^{m-1}\norm{\bar{\partial}_t \zeta^{n+1/2}}^2,  
\end{align*}
where { $N_{(f,u)}^2$} is defined in \eqref{eqn:N}. Then it is possible to apply the discrete {Gronwall} inequality from Lemma~\ref{P1 d-gronwall} to deduce the bound
\begin{equation}
\norm{\bar{\partial}_t \zeta^
{m+1/2}}^2+ \norm{\zeta^{m+1/2} }_h^2  \lesssim  h^{4\gamma_0}N_{(f,u)}^2
+{k^4} M_{(u)}^2. \label{l2cor}
\end{equation}
Finally, by virtue of triangle inequality, we can  show that 
\[\norm{\bar{\partial}_t u^{m+1/2}-\bar{\partial}_t U^{m+1/2} } + \norm{ u^{m+1/2}- U^{m+1/2} }_h  \le \| \rho^{m+1/2}  \|+ \| \bar{\partial}_t \rho^{m+1/2}  \|+\| \zeta^{m+1/2}  \|  + \| \bar{\partial}_t \zeta^{m+1/2}  \|.\] 
{The second-last} displayed estimate and Lemma \ref{P1 ritz_lemma} conclude the proof.
\end{proof}
\begin{rem}\label{newremark}
{The $L^2-$estimates $\norm{ u^{m+1}- U^{m+1} } \lesssim  h^{(2\gamma_0)} L_{(f,u)} +k^2 M_{(u)}$, $1\le m \le N-1,$ can be obtained using \eqref{l2cor} and similar arguments as in Corollary~\ref{corr}.} 
\end{rem}
\section{Numerical Experiments and Rates of Convergence}\label{P1 numeric_section}
In this section we conduct  computational tests that illustrate the convergence of the methods analyzed in the paper. All numerical routines have been realized using the open-source finite element libraries \texttt{FEniCS} \cite{AlnaesBlechta2015a} and \texttt{FreeFem++} \cite{Hecht2012}. For all the linear systems we use the direct solver MUMPS.

\subsection{Example 1: Convergence for a Smooth Solution}\label{conv}
The theoretical results of  Section~\ref{P1 fully_discrete_section} and Section~\ref{P1 implicit_direct} are validated in this section by choosing a smooth analytic solution of \eqref{P1 strong_form} and {using manufactured solutions.}  We consider the spatial domain $\Omega=(0,1)^2$ and time interval $[0,1]$. The data $f,\;u^0$ and $v^0$ are  chosen such that the exact solution is given by
\[u(x,t)=\exp(-t)(x_1(x_1-1)x_2(x_2-1))^2,\]
and hence our theoretical regularity assumptions (as well as the clamped boundary conditions) are satisfied. 

We construct a sequence of successively refined uniform triangular meshes of $\Omega$ and split the time domain {using a constant time step}. On each mesh refinement we compute errors between exact and approximate solutions, where the norms used to evaluate errors -- as well as  the penalty parameters chosen in case of dG and C$^0$IP schemes -- are specified in {Table~\ref{P1 norms}. We} also compute rates of convergence (with respect to the space discretization) according to $\texttt{Rate} =\log(\frac{e_{(\cdot)}}{\tilde{e}_{(\cdot)}})[\log(\frac{h}{\tilde{h}})]^{-1}$, where $e,\tilde{e}$ denote errors generated on two consecutive pairs of mesh size and time step ~$(h,k)$, and~$(\tilde{h},\tilde{k})$, respectively.  Table~\ref{P1 explicit_table} reports on the error decay and experimental convergence rates for the explicit leapfrog scheme. In order to adhere to the CFL condition, for each mesh refinement $\frac{k}{h^2}$ is chosen as $\frac{1}{100}$. 

\medskip 
In Table~\ref{P1 implicit_table}, we display the results of the convergence test using the implicit Newmark scheme, for which we choose $k=h$. The numerical results demonstrate the superiority of the implicit scheme over the explicit scheme--the absolute errors are smaller, CFL condition and requirement of quasi-uniformity of mesh are relaxed for this case.

\medskip \noindent 
In all cases, the  numerical results are consistent with the expected theoretical predictions in the sense that an experimental convergence order of ${\mathcal{O}}(h)$ and ${\mathcal{O}}(h^2)$ is observed for these methods for $l^{\infty}$ in time and {$\norm{\cdot}_h$ and $ \norm{\cdot}$} norms in space (as defined in Table~\ref{P1 norms}),  respectively. This is in line with the discussions in Sections~\ref{P1 fully_discrete_section} and \ref{P1 implicit_direct}. We also observe that for a given mesh resolution, the dG scheme gives the lowest $L^2-$error while the C$^0$IP scheme generates the lowest energy error. A sample of the approximate solution is plotted in Fig.~\ref{fig:conv}. 

\begin{figure}[t!]
    \centering
    \includegraphics[width=0.4\textwidth]{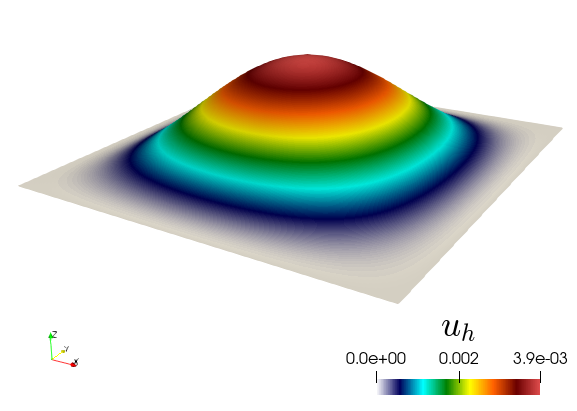}
    \caption{Example 1. Approximate solution computed with the C$^0$IP  {method} on a mesh with $h=0.031$.}
    \label{fig:conv}
\end{figure}

\begin{table}[t!]
\begin{center}
\begin{tabular} {| c |c |c|c|}\hline
        Scheme& $\underset{n \in \{{0,1,\cdots,N}\}}{\max}\|e_h^n\|$& $\underset{n \in \{{0,1,\cdots,N}\}}{\max}\|e_h^n\|_{h}$& Penalty parameters \\[0.5ex]\hline
Morley&$\underset{n \in\{ {0,1,\cdots,N}\}}{\max}\norm{u(t_n)-U^n}$&$\underset{n \in\{ {0,1,\cdots,N}\}}{\max}\trinr u(t_n)-U^n\trinr_{\text{\rm pw}}$&
No parameter\\
\hline      
  \vphantom{$\int^X_X$} dG & $\underset{n \in\{ {0,1,\cdots,N}\}}{\max}\norm{u(t_n)-U^n}$ & $\underset{n \in \{{0,1,\cdots,N}\}}{\max}\| u(t_n)-U^n\|_{\rm{dG}}$& $\sigma_{\rm dG}^1=10,\;\sigma_{\rm dG}^2=15$  \\ 
 \hline
 \vphantom{$\int^X_X$} $C^0$IP & $\underset{n \in\{ {0,1,\cdots,N}\}}{\max}\norm{u(t_n)-U^n}$ &$\underset{n \in \{{0,1,\cdots,N}\}}{\max}\| u(t_n)-U^n\|_{\rm{IP}}$&$\sigma_{\rm IP}=10$ \\ \hline
\end{tabular}
\end{center}
\vspace{-0.4cm}
\caption{Example 1. Error norms and penalty parameter values for different types of spatial discretizations.}
\label{P1 norms}
\end{table}

\begin{table}[t!]
\begin{center}
\begin{tabular} {| c c c c c c |}\hline
& $h$ & $L^2-$norm & \texttt{Rate} 
  & Energy norm & \texttt{Rate}  \\  \hline
Morley  & 0.250  & 3.01e-03 & --    & 8.75e-02 & -- \\ 
        & 0.125 & 7.95e-04 & 1.912 & 4.86e-02 & 0.793 \\
         & 0.062 & 2.12e-04 & 1.954 & 2.22e-02 & 0.898 \\
         & 0.031 & 5.33e-05 & 1.903 & 1.23e-02 & 0.931 \\
         & 0.016 & 1.41e-05 & 1.943  & 6.10e-03 & 0.965 \\ 
         \hline
 dG        & 0.250 & 2.90e-03 & -- &  1.45e-01 & -- \\
            & 0.125 & 5.36e-04 &  1.907  & 6.96e-02 & 0.782\\
           & 0.062 & 1.38e-04 & 1.958  & 4.36e-02 & 0.676\\
           & 0.031 & 3.26e-05 & 2.081  & 2.06e-02 & 1.080\\
           & 0.016 & 7.82e-06 & 2.059  & 8.51e-03 & 1.277 \\ 
           \hline
$C^0$IP    & 0.250 & 1.30e-03 & -- & 6.64e-02 & --\\
           & 0.125 & 5.12e-04 & 1.348  & 3.21e-02 & 1.051\\
           & 0.062 & 1.52e-04 & 1.749  & 1.50e-02 & 1.096\\
           & 0.031 & 4.04e-05 & 1.917  & 7.16e-03 & 1.066\\
           & 0.016 & 1.03e-05 & 1.974  & 3.49e-03 & 1.038\\ \hline 
\end{tabular}
\end{center}
\vspace{-0.4cm}
\caption{Example 1. Error history (errors in the $L^2$ and energy norms and estimated convergence rates) associated with the explicit version of the scheme for different discretizations in space.}\label{P1 explicit_table}
\end{table}

\begin{table}[t!]
\begin{center}
\begin{tabular} {|c c c c c c |}\hline
& $h$ & $L^2-$norm & \texttt{Rate} 
  & Energy norm & \texttt{Rate}  \\  \hline
Morley     & 0.250& 1.48e-03&--&4.51e-02&--\\
           & 0.125& 6.02e-04&1.301&3.76e-02&0.260\\
           & 0.062&2.00e-04&1.586&2.35e-02&0.677\\
           & 0.031&5.22e-05&1.942&1.30e-02&0.857\\ 
           & 0.016&1.32e-05&1.978&6.70e-03&0.956\\
           & 0.008&3.31e-06& 1.997&3.40e-03&0.975\\
            \hline      
 dG        & 0.250 & 8.39e-04 & --  & 8.99e-02 & --\\
           & 0.125 & 4.22e-04 & 0.992  & 6.91e-02 & 0.378\\
           & 0.062 & 1.32e-04 & 1.673  & 4.49e-02 & 0.625\\
           & 0.031 & 3.19e-05 & 2.050  & 2.20e-02 & 1.031\\
           & 0.016 & 7.79e-06 & 2.035  & 9.49e-03 & 1.211\\
           & 0.008 & 1.94e-06 & 2.006  & 4.27e-03 & 1.153\\ \hline
$C^0$IP    & 0.250 & 7.11e-04 & --  & 5.15e-02 & --\\
           & 0.125 & 4.20e-04 & 0.759  & 2.61e-02 & 0.979\\
           & 0.062 & 1.43e-04 & 1.557  & 1.39e-02 & 0.908\\
           & 0.031 & 3.94e-05 & 1.859  & 6.84e-03 & 1.025\\
           & 0.016 & 1.01e-05 & 1.967  & 3.40e-03 & 1.006\\
           & 0.008 & 2.58e-06 & 1.968  & 1.70e-03 & 1.003\\ \hline
\end{tabular}
\end{center}
\vspace{-0.4cm}
\caption{Example 1. Error history (errors in the $L^2$ and energy norms and experimental convergence rates) associated with the implicit version of the scheme for different discretizations in space.}
\label{P1 implicit_table}
\end{table}

\subsection{Example 2: Vibration in Heterogeneous Media}
The aim of this subsection is to illustrate that lowest-order nonstandard methods can effectively address a slightly more complex problem. For this we follow a similar approach as in \cite{MR4444402} and consider the following modified PDE 
\begin{equation*}
u_{tt}(x,t) +  \Delta (c(x) \Delta u(x,t)) = f(x,t), \quad (x,t) \in \Omega \times \left(0,T \right], 
\end{equation*}
with initial and clamped boundary conditions
\begin{align*}
\begin{cases}
\;u(x,0)=u^0(x),&\\
\hspace{0.01cm}u_t(x,0)=v^0(x) &\quad\text{in }\Omega,\\
 \; u=\frac{\partial u}{\partial n} =0, \quad\text{on }\partial \Omega \times \left(0,T \right].
  \end{cases}
\end{align*}
where the positive coefficient $c$ characterizes the rigidity of the material being considered. The setting adopted for the experiment is
\begin{equation*}
    \Omega=(-1,1)^2, \qquad T=\frac{3}{100}, \qquad c(x) = \begin{cases}
1, & \text{if } x_2 < 0.2, \\
9, & \text{if } x_2 \geq 0.2, 
\end{cases}\qquad 
f=0,
\end{equation*}
with the initial values 
\begin{equation*}
    u^0=\frac15\exp(-|10x|^2)[(1-x_1^2)^2(1-x_2^2)^2], \qquad v^0=0.
\end{equation*}

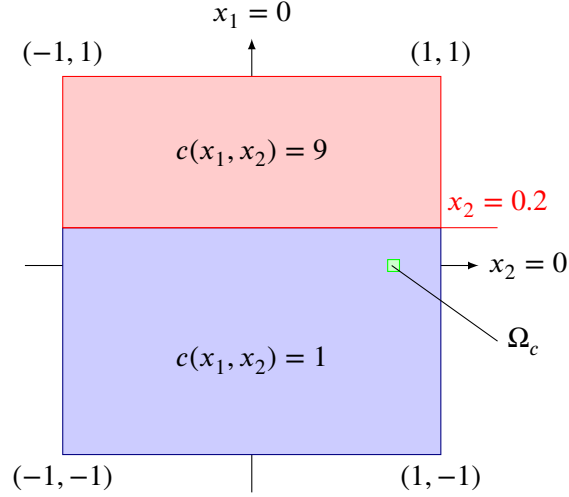
\begin{figure}[t!]
\centering
\begin{tikzpicture}[scale=2.5, >=latex]
    \draw[->] (-1.2,0) -- (1.2,0) node[right] {$x_2=0$};
    \draw[->] (0,-1.2) -- (0,1.2) node[above] {$x_1=0$};

    \filldraw[fill=blue!20, draw=blue!50!black] (-1,-1) rectangle (1,0.2);
    \filldraw[fill=red!20, draw=red!50!red] (-1,0.2) rectangle (1,1);
    \filldraw[fill=green!20, draw=green!50!green] (0.71875,-0.03125) rectangle (0.78125,0.03125);

    \draw[red] (-1,0.2) -- (1.3,0.2) node[above] {$x_2=0.2$};
    \draw[black] (0.74,0) -- (1.3,-0.4) node[right] {$\Omega_c$};
    \node at (0,0.6) {$c(x_1,x_2) = 9$};
    \node at (0,-0.5) {$c(x_1,x_2) = 1$};
    
    \node[below] at (-1,-1) {$(-1,-1)$};
    \node[below] at (1,-1) {$(1,-1)$};
    \node[above] at (-1,1) {$(-1,1)$};
    \node[above] at (1,1) {$(1,1)$};
\end{tikzpicture}
\caption{Example 2. Piecewise constant function $c(x_1,x_2)$ in the domain $\Omega = (-1,1)^2$.}
 \label{cxx}
\end{figure}
The initial value is a regularized Dirac impulse and stimulates the system's dynamics. As in the aforementioned reference, we define the control region $\Omega_c =(0.75-l_c,0.75+l_c) \times (-l_c,l_c)$ with $l_c=\frac{1}{32}$ (see the small green shaded region in Figure~\ref{cxx}) that simulates a sensor and evaluate the expression 
$u_c(t)=\int_{\Omega_c} u_{h}(x,t) \dx.$
Our focus is to examine and compare the signal arrival at the sensor position for different {grid and time step sizes} for all the three lowest-order schemes discussed in Section~\ref{conv}.  {The calculations were performed on fixed 50×50 and 100×100 grids, where each square of the grid is divided into two triangles along the diagonal from the bottom-left to the top-right in the triangulation}. The time steps we chosen as $k=\frac{T}{N}$, where $N$ denotes the number  {of} time subintervals. The numerically computed control quantity is presented in  Fig. \ref{moruct} (for Morley), Fig.~\ref{dguct} (for dG) and Fig.~\ref{c0ipuct} (for C$^0$IP). Similar graphs can be seen for different time step sizes for corresponding lowest-order schemes which shows the accuracy of schemes. For reference we also plot a warp of the domain by the scalar solution in Fig.~\ref{fig:3d}.

\begin{figure}[t!]
\centering
\subfloat{ \includegraphics[width =8cm, height=6cm]{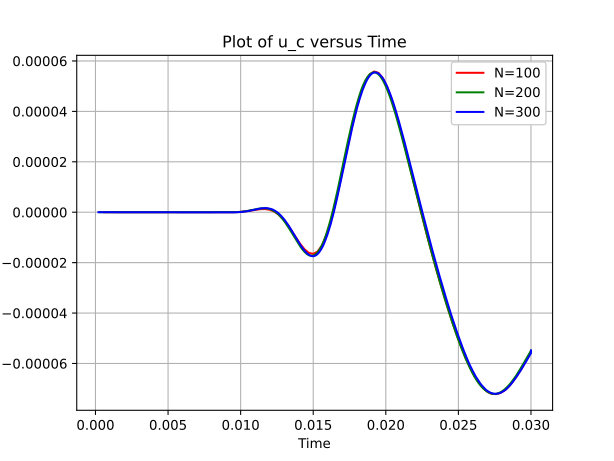}
}
\subfloat{
\includegraphics[width=8cm, height=6cm]{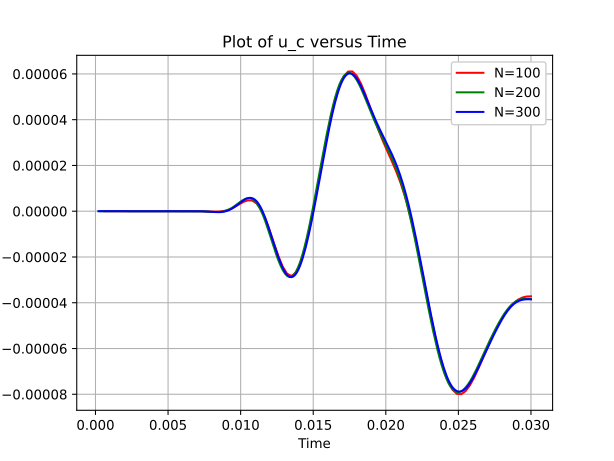}
}
\caption{Example 2. Output quantity $u_c(t)$ on a $50\times 50 $ (left) and $100\times 100$ (right) grid, computed with the Morley scheme.} \label{moruct}
\end{figure}
\begin{figure}[t!]
\centering
\subfloat{ \includegraphics[width =8cm, height=6cm]{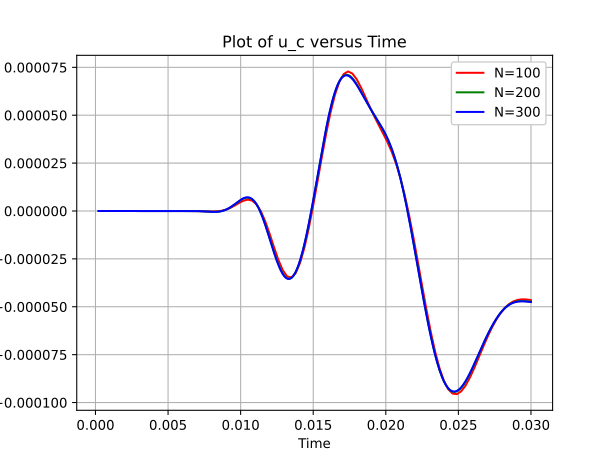}
}
\subfloat{
\includegraphics[width=8cm, height=6cm]{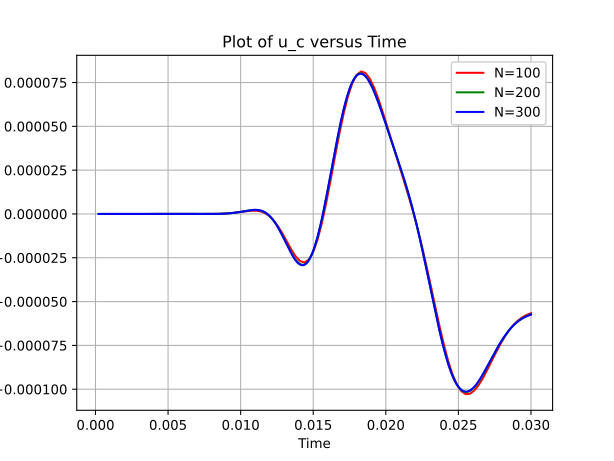}
}
\caption{Example 2. Output quantity $u_c(t)$ on a $50\times 50 $ (left) and $100\times 100$ (right) grid, computed with the dG scheme with $\sigma^1_\text{dG}=\sigma^2_\text{dG}=20$.} \label{dguct}
\end{figure}
\begin{figure}[t!]
\centering
\subfloat{ \includegraphics[width =8cm, height=6cm]{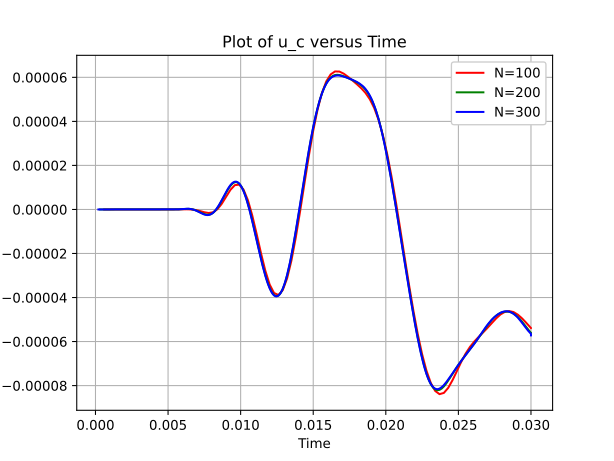}
}
\subfloat{
\includegraphics[width=8cm, height=6cm]{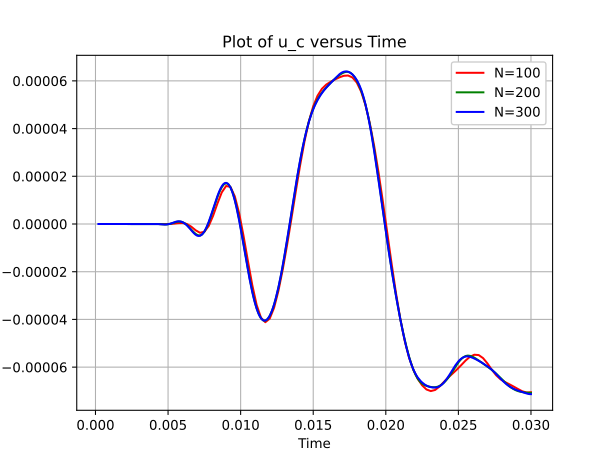}
}
\caption{Example 2. Output quantity $u_c(t)$ on a $50\times 50 $ (left) and $100\times 100$ (right) grid, computed with the C$^0$IP scheme with $\sigma_\text{IP}=10$.} \label{c0ipuct}
\end{figure}

\begin{figure}[t!]
    \centering
    \includegraphics[width=0.325\textwidth]{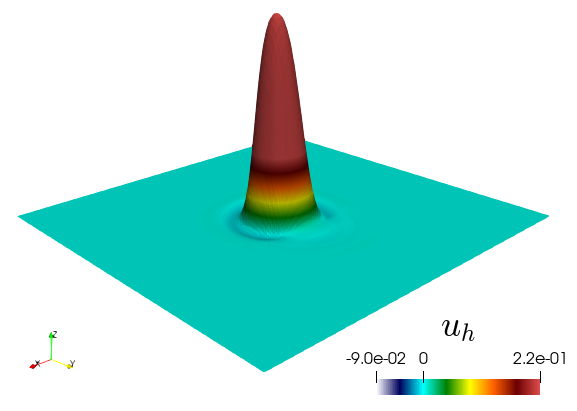}
    \includegraphics[width=0.325\textwidth]{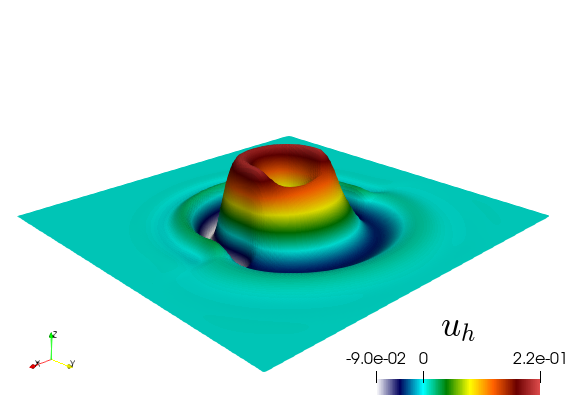}
    \includegraphics[width=0.325\textwidth]{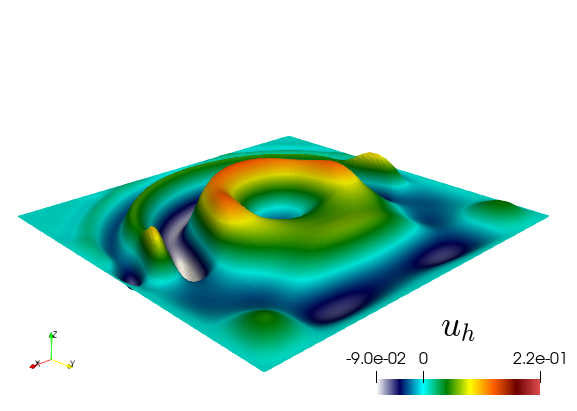}
    \caption{Example 2. Snapshots of the approximate solution computed with the dG scheme, and shown at three different time steps.}
    \label{fig:3d}
\end{figure}
\noindent

\medskip \noindent 
\textbf{Conclusions}

\medskip \noindent 
This article provides a comprehensive analysis of the biharmonic wave equation, combining nonstandard finite element methods for spatial discretization with explicit and implicit schemes for temporal discretization. A CFL condition specifically tailored to the biharmonic wave equation is derived. Though the explicit scheme for time discretization is attractive, we have found that it requires a quasi-uniform mesh, and in turn a stringent CFL condition. On the other hand, the implicit schemes discussed in this paper relax both requirements. A modified Ritz projection and a novel lifting technique provide a different perspective to the error analysis and facilitate error estimates under sharper regularity assumptions on the data in comparison to the results available in the literature.
The study demonstrates quasi-optimal convergence for semidiscrete and fully-discrete schemes, and this provides a robust foundation for problems for future study; for example,  nonlinear plate problems \cite{MR2643040}, coupled multiphysics (poroelastic \cite{MR3575720}, thermoelastic \cite{MR1458525},  thermoelastic-diffusion \cite{MR3362761})   thin plate models. We also mention that a posteriori error analysis and adaptivity for this family of schemes is a direction we will explore in future work. \\
\textbf{Funding:} This research has been partially supported by the Australian Research Council through the Discovery
Project grant DP220103160 and the Future Fellowship grant FT220100496.
\appendix
\section{Appendix}
\subsection{Error Analysis of Explicit Scheme Using Semidiscrete Error Bounds}\label{p1 explicit_fully_semi_section}
In this subsection error estimates for the explicit scheme are obtained by using semidiscrete estimates rather than a direct approach used in {Sub}section~\ref{Explicit_direct_section}. 
We split the error as 
\begin{equation*}
    u(t_n)-U^n =(u(t_n)-{\cal R}_h u(t_n) )+({\cal R}_h u(t_n)-u_h(t_n))+(u_h(t_n)-U^n):= \rho^n+\theta ^n+ \chi^n .
    \end{equation*}
    Recall  ${\widetilde \tau}^n$ is defined in Lemma~\ref{tau}. A combination of \eqref{P1 semi_formulation} and \eqref{P1 explicit_fully_discrete} yields the error equation in $\chi^n$ as
    \begin{equation} (\partial^2_t \chi^n,v_h)+a_h (\chi^n,v_h)=({\widetilde \tau}^n,v_h )\text{ for all }v_h \in V_h. \label{P1 explicit_semi_full_error_eqn}
    \end{equation}
Recall that $\phi^{1/2}=\frac{1}{2}(\phi^0+\phi^1)$. The next lemma establishes the bounds on initial error $\chi^{1/2}:=u_h^{1/2}-U^{1/2} $. 

\begin{lemma}[Initial Error Bounds] \label{P1 semi_fully_lemma_on_norm_initial}
{The initial error $\chi^{1/2}:=u_h^{1/2}-U^{1/2} $ satisfies}
\begin{align*}
    \norm{\bar{\partial}_t \chi^{1/2}}^2 +\norm{ \chi^{1/2}}_h^2
    \lesssim
    {k^4}\norm{u_{httt}}^2_{L^ {\infty}(L^2(\Omega)) }+h^{4\gamma_0}\norm{v^0}^2_{{H^{2+\gamma_0}(\Omega)}},
\end{align*}
where the constant  in ``$\lesssim$'' depends on  $\alpha$ from \eqref{P1 a_h_properties} and $C_2$ from \eqref{P1 ritz_lemma}.
\end{lemma}
\begin{proof}[\textbf{Proof}]
For any $v_h \in V_h$, the formulations \eqref{P1 semi_formulation} and \eqref{P1 explicit_ic_1} show 
\begin{align}
 2k^{-1} (\bar{\partial}_t \chi^{1/2},v_h)+ a_h( \chi^{1/2},v_h)
 &=2k^{-1} (\bar{\partial}_t u_h^{1/2},v_h)+ a_h( u_h^{1/2},v_h)-2k^{-1} (\bar{\partial}_t U^{1/2},v_h)- a_h( U^{1/2},v_h)
\nonumber \\&= 2k^{-1}(\bar{\partial}_t u_h^{1/2},v_h)+a_h(u_h^{1/2},v_h)-(  f^{1/2} +2k^{-1}v^0,v_h  ) \nonumber \\
 &=
({\widetilde R}^0,v_h)+2k^{-1}(\rho_t^0,v_h) \label{chi_0}
\end{align}
with ${\widetilde R}^0:=2k^{-1}(\bar{\partial}^2_t u_h^{1/2}-u_{ht}^0) -u_{htt}^{1/2}$ from Lemma~\ref{R0} and $\rho_t^0=u_{ht}^0-v^0={\cal R}_hv^0-v^0$.\\
{Choose $v_h= \chi^{1/2}$} in \eqref{chi_0} and utilize $ 2k^{-1}\chi^{1/2}=\bar{\partial}_t\chi^{1/2}$ (since $\chi^0={\cal{R}}_h u^0-U^0=0$) to obtain
\begin{align}
(\bar{\partial}_t \chi^{1/2},\bar{\partial}_t\chi^{1/2} )+ a_h( \chi^{1/2},\chi^{1/2})=
\frac{1}{2}k({\widetilde R}^0,\bar{\partial}_t \chi^{1/2})+(\rho_t^0,\bar{\partial}_t\chi^{1/2}) .\nonumber
\end{align}
The ellipticity of $a_h(\cdot,\cdot)$ from \eqref{P1 a_h_properties} and Cauchy--Schwarz  inequality reveal
\begin{align}
\norm{\bar{\partial} _t \chi^{1/2}}^2+\alpha \norm{\chi^{1/2}}_h^2 &\le \frac{1}{2}k \norm{{\widetilde R}^0}\norm{\bar{\partial}_t \chi^{1/2}}+\norm{\rho_t^0}\norm{\bar{\partial}_t\chi^{1/2}}. \label{Ellip}
\end{align}
The Young's inequality (applied twice) with $a=k\norm{{\widetilde R}^0}$ (resp. $a=\norm{\rho_t^0}$), $b=\norm{\bar{\partial}_t \chi^{1/2}}$ (resp. $b=\norm{\bar{\partial}_t \chi^{1/2}}$), $\epsilon=1$ (resp. $\epsilon=2$) for first (resp. second) term on the right-hand side of \eqref{Ellip} show
\begin{align}
     \frac{1}{2}k \norm{{\widetilde R}^0}\norm{\bar{\partial}_t \chi^{1/2}}+\norm{\rho_t^0}\norm{\bar{\partial}_t\chi^{1/2}} \le \frac{1}{4}k^2\norm{{\widetilde R}^0}^2+\norm{\rho_t^0}^2+\frac{1}{2}\norm{\bar{\partial}_t \chi^{1/2}}^2. \label{youn}
\end{align}
A combination of \eqref{Ellip} and \eqref{youn} with bounds for ${\widetilde R}^0$ from Lemma~\ref{R0} and $\rho_t^0$ from Lemma~\ref{P1 ritz_lemma} conclude the proof.
\end{proof}
\noindent
 The next theorem gives the error bounds under the CFL condition on mesh ratio discussed in  Theorem~\ref{P1 explicit_stability}.  
 
 \begin{thm}[Error Estimates]\label{Th using semi} {Consider a quasi-uniform {triangulation ${\mathcal{T}}$} of $\bar{\Omega}$ and assume that the regularity results in Lemma \ref{P1 regularity_lemma} and CFL condition $k \le {\beta^{-1/2}C^{-1}_{\text{\rm inv}}} h^2$, hold.} {Then for $1 \le m\le N-1$,}
 \begin{align*} \norm{\bar{\partial}_t u^{m+1/2}-\bar{\partial}_t U^{m+1/2} } + \norm{ u^{m+1/2}- U^{m+1/2} }_h \lesssim  h^{\gamma_0} L_{(f,u)} +k^2 M_{(u_h)}, 
\end{align*}
where $L_{(f,u)}$ is given in \eqref{L},  
$M_{(u_h)}:=  \norm{u_{httt}}_{L^\infty {(L^2(\Omega))}} +\norm{u_{htttt}}_{L^2{(L^2(\Omega))}},$ and the  constant absorbed in "$\lesssim$'' depends on $\alpha$, $\beta$ from \eqref{P1 a_h_properties},  $T$,  and $C_1$, $C_2$ from Lemmas~\ref{bhcompanion_lem}-\ref{P1 ritz_lemma}.
    \end{thm}
    \begin{proof}[Proof]
     Choose $v_h= 2k \delta_t \chi ^n=2(\chi^{n+1/2}-\chi^{n-1/2})=k(\bar{\partial}_t \chi^{n+1/2}+\bar{\partial}_t \chi^{n-1/2})$ in \eqref{P1 explicit_semi_full_error_eqn} and repeat the arguments in Theorem~\ref{P1 explicit_stability} with $U^n$ replaced by  $\chi^n$ and $f^n$  by ${\widetilde{\tau}}^n$ to arrive at 
 \begin{equation*} \norm{\bar{\partial}_t \chi^{m+1/2}}^2+ \norm{\chi^{m+1/2}}_h^2\lesssim \norm{\bar{\partial}_t \chi^{1/2}}^2+\norm{\chi^{1/2}}_h^2+ Tk\sum_{n=1}^m\norm{{\widetilde \tau}^n}^2 +k/T \sum_{n=0}^{m-1} \norm{\bar{\partial}_t \chi^{n+1/2}}^2.
\end{equation*} 
 The first two terms on the right-hand side of the above expression are bounded using Lemma \ref{P1 semi_fully_lemma_on_norm_initial} and the truncation error $\widetilde{\tau}^n$ is bounded using  Lemma~\ref{tau}. Then we apply the discrete {Gronwall} Lemma~\ref{P1 d-gronwall} to deduce
\begin{equation*} \norm{\bar{\partial}_t \chi^{m+1/2}}^2+ \norm{\chi^{m+1/2}}_h^2\lesssim k^4 \norm{u_{httt}}^2_{L^\infty(L^2(\Omega))}+ k^4 \norm{u_{htttt}}^2_{L^2(L^2(\Omega))}+h^{4\gamma_0}\norm{v^0}^2_{H^{2+\gamma_0}(\Omega)}.
    \end{equation*}Then,  triangle inequality readily shows that  
    \begin{align*}
    \norm{\bar{\partial}_t u^{m+1/2}-\bar{\partial}_t U^{m+1/2} } + \norm{ u^{m+1/2}- U^{m+1/2} }_h  &\le \| \rho^{m+1/2}  \| + \| \theta ^{m+1/2}\| \\
    &\quad+ \|\chi^{m+1/2} \| + \| \bar{\partial}_t \rho^{m+1/2}  \| + \| \bar{\partial}_t \theta ^{m+1/2}\| +   \|\bar{\partial}_t \chi^{m+1/2} \|.
     \end{align*}
Therefore the proof is completed after using the last displayed estimate,  Theorem \ref{P1 semi_error_estimates}, and Lemma \ref{P1 ritz_lemma}.
\end{proof}
\subsection{Regularity}
This subsection is devoted to the proof of Lemma \ref{P1 regularity_lemma} {following the approach from \cite[Theorem 12.3]{chen}.} 
First, we recall that  $\left\{\psi_n \right\}_{n=1}^\infty$ is an  orthonormal basis of $L^2(\Omega)$, and hence (cf. \cite[Chapter 1, Theorem 4.13]{conway}) 
\begin{equation}
   u(x,t) = \sum_{n=1}^{\infty} d_n(t)\psi_n\text{ and  }  u_{tt}(x,t)= {\sum_{n=1}^{\infty} {d_{ntt}}\psi_n}, \label{P1 representation}
\end{equation}
where $d_n(t):= ( u(t), \psi_n)$ and ${d_{ntt}(t)}=( u_{tt}(t), \psi_n)$. A combination of \eqref{P1 strong_form} and \eqref{P1 representation} and the fact that the $\psi_n$'s are smooth functions satisfying \eqref{P1 Ev} reveal that for every $j \ge 1$,
\begin{equation*}
     (f(t), \psi_j)=\sum_{n=1}^\infty \left({d_{ntt}(t)}\psi_n , \psi_j \right)+  \sum_{n=1}^\infty \left( \Delta^2  d_n(t)\psi_n , \psi_j \right)=\sum_{n=1}^\infty {d_{ntt}(t)}\left(\psi_n , \psi_j \right)+ \sum_{n=1}^\infty  d_n(t)\left( \lambda_n \psi_n , \psi_j \right).
\end{equation*}
The orthonormality of the $\psi_n$'s simplifies the above equation to a second-order linear {ODE} 
$${d_{jtt}}+\lambda_j d_j=f_j$$
with 
\begin{align*}
   & 
   d_j(0)=(u(0),\psi_j)=(u^0,\psi_j),\;  
   {d_{jt}(0)}=(u_t(0),\psi_j)=(v^0,\psi_j), \;\text{and }  f_j(t)=(f(t), \psi_j),
\end{align*}
for all $t\in[0,T]$.
For $n \ge 1$, the solution of this ODE is 
\begin{equation}
    d_n(t)=(u^0, \psi_n){\cos (\sqrt{\lambda_n}t)}+ \lambda_n^{-\frac{1}{2}}{(v^0, \psi_n)}{\sin (\sqrt{\lambda_n} t )}+ \lambda_n^{-\frac{1}{2}}I(t), \label{P1 dj}
\end{equation}
\noindent 
where
$ \displaystyle
 I(t):=\int_0 ^ t {\sin(\sqrt{\lambda_n}(t-s))} f_n(s) \ds .
$
A successive differentiation with respect to $t$ shows that 
\begin{align*}
    {I_t(t)}&=\sqrt{\lambda_n}\int_0 ^ t {\cos (\sqrt{\lambda_n}(t-s))} f_n(s) \ds \text{ and }
    {I_{tt}(t)}=\sqrt \lambda_n f_{n}(t)-{\lambda_n}\int_0 ^ t {\sin( \sqrt{\lambda_n}(t-s) )}f_n(s)\ds.
\end{align*}
Then we can apply integration by parts thrice to the term {$ I_{tt}(t)$   to observe} 
\begin{subequations}\label{app:i2}
\begin{align}
{I_{tt}(t)}&=\sqrt \lambda_n f_n(0){\cos(\sqrt \lambda_n t)}+\sqrt \lambda_n\int_0^t {f_{ns}(s)}{\cos(\sqrt \lambda_n (t -s))}\ds, \label{I'' 1}\\
    {I_{tt}}(t)&=\sqrt{\lambda_n}f_n(0){\cos (\sqrt{\lambda_n}t)}+{f_{nt}(0)}{\sin (\sqrt{\lambda_n}t)}+\int_0^t {f_{nss}(s)}{\sin (\sqrt{\lambda_n}(t-s))}\ds,\\
    {I_{tt}(t)} &=\sqrt{\lambda_n}f_n(0){\cos (\sqrt{\lambda_n}t)}+{f_{nt}(0)} {\sin (\sqrt{\lambda_n}t )}-{\lambda_n^{-\frac{1}{2}}}{f_{ntt}(0)}{\cos( \sqrt{\lambda_n}t)} \nonumber\\
    & \qquad +\lambda_n^{-\frac{1}{2}}{f_{ntt}(t)}-\lambda_n^{-\frac{1}{2}} \int_0 ^ t {\cos (\sqrt{\lambda_n}(t-s) )}{f_{nsss}(s)} \ds.\label{P1 I''}
\end{align}\end{subequations}
{The expressions in \eqref{app:i2} are utilized appropriately to control $u_{tt}$ in the $L^2(\Omega)$ and $H^{2+\gamma_0}(\Omega)$ norms. In addition, the integration by parts is aimed at reducing} the spatial regularity of $f$ and its time derivatives.
Next, we proceed to 
differentiate \eqref{P1 dj} twice and use \eqref{P1 norm_DA} (with $D(A^0)=L^2(\Omega)$), which leads to  
\begin{align}
\norm{u_{tt}}^2 = \sum_{n=1}^ \infty|(u_{tt},\psi_n)|^2& =\sum_{n=1}^ \infty|{d_{ntt}}|^2 \nonumber\\
&
=\sum_{n=1}^ \infty | -\lambda_n (u^0, \psi_n){\cos (\sqrt{\lambda_n}t)}-  \sqrt{\lambda_n}{(v^0, \psi_n)}{\sin (\sqrt{\lambda_n} t )}+{{\lambda_n}}^{-\frac{1}{2}}{I_{tt}(t)}|^2.\label{l2utt}
\end{align}
A combination of  this  with \eqref{I'' 1} and a use of the definitions of $f_n, {f_{ns}}$ with elementary manipulations show
\begin{align}
\norm{u_{tt}}^2&\lesssim \sum_{n=1}^ \infty \left( |\lambda_n (u^0, \psi_n)|^2+|\sqrt{\lambda_n }(v^0, \psi_n)  |^2 +| f_n(0){\cos (\sqrt{\lambda_n}t)}+\int_0 ^ t {f_{ns}(s)} {\cos( \sqrt{\lambda_n}(t-s) )} \ds|^2 \right) \nonumber\\
&\lesssim \sum_{n=1}^ \infty \left( |\lambda_n (u^0, \psi_n) |^2+|\sqrt{\lambda_n }(v^0, \psi_n)  |^2 +\left| (f(0),\psi_n)\right|^2+\int_0 ^ t |({f_{s}(s)},\psi_{n})|^2 \ds \right). \label{utt}
\end{align}
Since $f_t \in L^2(\Omega)$, an application of the monotone convergence theorem and   \eqref{P1 norm_DA} shows 
\begin{equation}    \norm{u_{tt}}^2_{L^\infty(L^2(\Omega))}\lesssim \left( \norm{u^0}^2_{D(A)}+\norm{v^0}^2_{D(A^{1/2})}+\norm{f(0)}^2+\norm{{f_t}}^2_{L^2(L^2(\Omega))}\right) \label{l2utt norm}.
\end{equation}
Then we differentiate  \eqref{P1 dj} thrice (resp. four times) to obtain 
\begin{align*}
    {d_{nttt}(t)}&=\lambda_n^{\frac{3}{2}}(u^0,\psi_n){\sin(\sqrt{\lambda_n}t)}-\lambda_n (v^0,\psi_n) {\cos (\sqrt{\lambda_n}t)}+{{\lambda_n}}^{-\frac{1}{2}}{I_{ttt}(t)}\\
\big(\text{resp. } {d_{ntttt}(t)}&=\lambda_n^{2}(u^0,\psi_n){\cos (\sqrt{\lambda_n}t)}+\lambda_n^{\frac{3}{2}} (v^0,\psi_n){\sin(\sqrt{\lambda_n}t)}+{{\lambda_n}}^{-\frac{1}{2}}{I_{tttt}(t)}\big).
\end{align*}
{An integration} by parts twice (resp. thrice) to the last term ${I_{ttt}(t)}$ (resp. ${I_{tttt}(t)}$) leads to
\begin{align*}
{d_{nttt}(t)}&=\lambda_n^{\frac{3}{2}}(u^0,\psi_n){\sin(\sqrt{\lambda_n}t)}-\lambda_n (v^0,\psi_n) {\cos (\sqrt{\lambda_n}t)}-\sqrt{\lambda_n}{\sin( \sqrt{\lambda_n}t)} f_{n}(0)\\
&\quad +{\cos (\sqrt{\lambda_n}t)} {f_{nt}(0)}+\int_0^t {\cos (\sqrt{\lambda_n}(t-s))}{f_{nss}(s)}\ds\\
\big(\text{resp. } {d_{ntttt}(t)}&=\lambda_n^{2}(u^0,\psi_n){\cos (\sqrt{\lambda_n}t)}+\lambda_n^{\frac{3}{2}} (v^0,\psi_n){\sin(\sqrt{\lambda_n}t)}-\lambda_n{\cos (\sqrt{\lambda_n}t)}f_n(0)\\
&\quad -\sqrt{\lambda_n}{\sin(\sqrt{\lambda_n}t)}{f_{nt}(0)}+{\cos (\sqrt{\lambda_n})}{f_{ntt}(0)}+\int_0^t {\cos (\sqrt{\lambda_n}(t-s))}{f_{nsss}(s)}\ds\big).
\end{align*}
The fact that $\norm{u_{ttt}}^2=\sum_{n=1}^ \infty|{d_{nttt}(t)}|^2$ (resp. $\norm{u_{tttt}}^2=\sum_{n=1}^ \infty|{d_{ntttt}(t)}|^2$) from \eqref{P1 norm_DA} and an approach similar to \eqref{l2utt}-\eqref{l2utt norm} leads to 
\begin{align*}
\norm{u_{ttt}}^2_{L^\infty(L^2(\Omega))}&\lesssim \norm{u^0}^2_{D(A^{3/2})}+\norm{v^0}^2_{D(A)}+\norm{f(0)}^2_{D(A^{1/2})}+\norm{{f_t(0)}}^2 +\norm{{f_{tt}}}^2_{L^2 (L^2(\Omega))},\\ \norm{u_{tttt}}^2_{L^\infty(L^2(\Omega))} &\lesssim \norm{u^0}^2_{D(A^2)}+\norm{v^0}^2_{D(A^{3/2})}+\norm{f(0)}^2_{D(A)}+\norm{{f_t(0)}}_{D(A^{1/2})}^2+\norm{{f_{tt}(0)}}^2+\norm{{f_{ttt}}}^2_{L^2 (L^2(\Omega))}.
\end{align*}
Now we aim to control $\|u_{tt}\|_{D(A)}$. A differentiation of \eqref{P1 dj} twice and substitution of ${d_{ntt}}$ in \eqref{P1 representation} shows
\begin{align}  \norm{u_{tt}}^2_{D(A)}=\sum_{n=1}^ \infty|\lambda_n {d_{ntt}(t)}|^2=\sum_{n=1}^ \infty \left| -\lambda_n ^2(u^0, \psi_n){\cos( \sqrt{\lambda_n}t)}-  \lambda_n^{\frac{3}{2}}{(v^0, \psi_n)}{\sin( \sqrt{\lambda_n} t)} +{\lambda_n^{\frac{1}{2}}} {I_{tt}(t)}\right|^2 .\label{uth33}
    \end{align} 
We can then argue as before by using \eqref{P1 I''} to conclude that, for any $t \in [0,T],$
\begin{align*}
\norm{u_{tt}}^2_{L^\infty({H^{2+ \gamma_0}(\Omega))}}&\lesssim  \sum_{n=1}^ \infty \left(  |\lambda_n ^2(u^0, \psi_n) |^2 + \left|   \lambda_n^{\frac{3}{2}}{(v^0, \psi_n)}  \right|^2 +|\lambda_nf_n(0) |^2 +|\sqrt{\lambda_n} {f_{nt}(0)}|^2 \right.\\
 &\left. \qquad \qquad +|{{f_{ntt}(0)}}|^2+|{f_{ntt}(t)}|^2+ \int_0 ^ t  |{f_{nsss}(s)} |^2
    \ds  \right).
\end{align*}
As $f,\;{f_t,\;f_{tt}, \text{ and } f_{ttt}} \in L^2(L^2(\Omega))$, from the Sobolev embedding we can infer that $ f,\;{f_t,\text{ and } f_{tt}}\in L^\infty(L^2(\Omega))$. 
Consequently, \eqref{P1 norm_DA} and the fact that $D(A) \subset H^{2+\gamma_0}(\Omega)$ (i.e., $\norm{u_{tt}}^2_{H^{2+ \gamma_0}(\Omega)}{\lesssim} \norm{u_{tt}}^2_{D(A)})$ show that  $\norm{u_{tt}}^2_{H^{2+ \gamma_0}(\Omega)} $ is bounded (up to a multiplicative constant) by
\begin{equation}
\norm{u^0}^2_{D(A^2)}+\norm{v^0}^2_{D(A^{3/2})}+\norm{f(0)}^2_{D(A)}+\norm{{f_t(0)}}^2_{D(A^{1/2})} +\norm{{f_{tt}(0)}}^2+\norm{{f_{tt}}}^2_{L^\infty(L^2(\Omega))}+\norm{{f_{ttt}}}^2_{L^2(L^2(\Omega))}.\label{uth3}
\end{equation}
{{Consider \eqref{P1 dj} with integration by parts applied to  the last term $I(t)$ once to obtain
\begin{align*}
d_n(t)&=(u^0, \psi_n){\cos (\sqrt{\lambda_n}t)}+ \lambda_n^{-\frac{1}{2}}{(v^0, \psi_n)}{\sin (\sqrt{\lambda_n} t)} +\lambda_n^{-1}{f_n(t)}\\
&\quad -\lambda_n^{-1}{\cos (\sqrt{\lambda_n}t)}f_n(0)-\lambda_n^{-1}\int_0^t {\cos (\sqrt{\lambda_n}(t-s))}{f_{ns}(s)}\ds.
\end{align*}
A differentiation of \eqref{P1 dj} once followed by integration by parts of the term  ${I_t(t)}$ appearing in the derivative  twice, yields
\begin{align*}
 {d_{nt}(t)}&=-\sqrt{\lambda_n}(u^0, \psi_n){\sin (\sqrt{\lambda_n}t)}+ {(v^0, \psi_n)}{\cos(\sqrt{\lambda_n} t)}+\lambda_n^{-\frac{1}{2}}{\sin(\sqrt{\lambda_n} t) }f_n(0)\\
&\quad +\lambda_n^{-1}{f_{nt}(t)}-\lambda_n^{-1}{\cos (\sqrt{\lambda_n}t)}{f_{nt}(0)}-\lambda_n^{-1}\int_0^t {\cos (\sqrt{\lambda_n}(t-s))}{f_{nss}(s)}\ds. 
\end{align*}
The last two equations and an  analogous to the one used to derive \eqref{uth3} from \eqref{uth33}, shows that }}
\begin{align*}
\norm{u}^2_{L^\infty({H^{2+ \gamma_0}(\Omega))}}&\lesssim \norm{u^0}^2_{D(A)}+\norm{v^0}^2_{D(A^{1/2})}+\norm{f(0)}^2+\norm{f}^2_{L^\infty (L^2(\Omega))}+\norm{{f_t}}^2_{L^2 (L^2(\Omega))},\\ 
\norm{u_t}^2_{L^\infty({H^{2+ \gamma_0}
(\Omega))}}& \lesssim \norm{u^0}^2_{D(A^{3/2})}+\norm{v^0}^2_{D(A)}+\norm{f(0)}^2_{D(A^{1/2})}+\norm{{f_t(0)}}^2
+\norm{{f_t}}^2_{L^\infty(L^2(\Omega))}+\norm{{f_{tt}}}^2_{L^2(L^2(\Omega))}.
\end{align*}
{This concludes the proof.} \qed

\end{document}